\documentclass{amsart}
\usepackage{amssymb}
\usepackage{amsmath}
\usepackage{hyperref}

\newtheorem{thm}{Theorem}[section]
\newtheorem{co}[thm]{Corollary}
\newtheorem{pro}[thm]{Proposition}
\newtheorem{lem}[thm]{Lemma}

\theoremstyle{definition}

\newtheorem{df}[thm]{Definition}
\newtheorem{rem}[thm]{Remark}

\newcommand{\Ol}{\mathsf{O}}
\newcommand{\ol}{\mathsf{o}}
\newcommand{\SD}{\mathsf{SD}}
\newcommand{\Dav}{\mathsf{D}}
\newcommand{\dav}{\mathsf{d}}
\newcommand{\Das}{\mathsf{D}^{\ast}}

\newcommand{\ord}{\mathrm{ord}}

\newcommand{\sequ}[1]{{}^{#1}\!\mathcal{F}}
\newcommand{\freesequ}[1]{{}^{#1}\!\mathcal{A}^{\ast}}
\newcommand{\minsequ}[1]{{}^{#1}\!\mathcal{A}}

\DeclareMathOperator{\supp}{supp}
\DeclareMathOperator{\cm}{cm}

\numberwithin{equation}{section}

\begin{document}

\title{On the Olson and the Strong Davenport constants}

\author[O.~Ordaz, A.~Philipp, I.~Santos, and W.~A.~Schmid]{Oscar Ordaz, Andreas Philipp, Irene Santos and Wolfgang A.\ Schmid}
\address[O.~Ordaz and I.~Santos]{Departamento de Matem\'aticas y Centro ISYS, Facultad de Ciencias,
Universidad Central de Venezuela, Ap. 47567, Caracas 1041-A,
Venezuela.}
\email{oscarordaz55@gmail.com}
\address[A.~Philipp]{Institut f\"ur Mathematik und Wissenschaftliches Rechnen,
Karl-Fran\-zens-Universit\"at Graz,
Heinrichstra\ss e 36,
8010 Graz, Austria.}
\email{andreas.philipp@uni-graz.at}
\address[W.~A. Schmid]{CMLS, {\'E}cole polytechnique, 91128 Palaiseau cedex, France.}
\email{wolfgang.schmid@math.polytechnique.fr}
\thanks{W.~S. is supported by the Austrian Science Fund (FWF): J 2907-N18.}

\begin{abstract}
A subset $S$ of a finite abelian group, written additively, is called zero-sumfree if the sum of the elements of each non-empty subset of $S$ is non-zero.
We investigate the maximal cardinality of zero-sumfree sets, i.e., the (small) Olson constant.
We determine the maximal cardinality of such sets for several new types of groups; in particular, $p$-groups with large rank relative to the exponent, including all groups with exponent at most five. These results are derived as consequences of more general results, establishing new lower bounds for the cardinality of zero-sumfree sets for various types of groups.
The quality of these bounds is explored via the treatment, which is computer-aided, of selected explicit examples.
Moreover, we investigate  a closely related notion, namely the maximal cardinality of minimal zero-sum sets, i.e., the Strong Davenport constant.
In particular, we determine its value for elementary $p$-groups of rank at most $2$, paralleling and building on recent results on this problem for the Olson constant.
\end{abstract}

\maketitle

\section{Introduction}
\label{sec_intro}

A subset $S$ of a finite abelian group $(G,+,0)$ is called zero-sumfree
if the sum of the elements of each non-empty subset of $S$ is not the zero-element of $G$.
And, a set $S$ is said to have a (non-empty) zero-sum subset if it is not zero-sumfree.

It is a classical problem, going back to Erd\H{o}s and Heilbronn \cite{heilbronn}, to determine the smallest integer $\ell_{G}$ such that each subset $S$ of $G$ with $|S|\ge \ell_{G}$ has a zero-sum subset; or, equivalently (the values of course differ by $1$) to determine the maximal cardinality of a zero-sumfree set.
Now, it is quite common to call this $\ell_G$ the Olson constant of $G$, denoted $\Ol(G)$; this name was introduced
by Ordaz, in 1994 during a seminar held at the Universidad Central de Venezuela (Caracas), as a tribute to Olson's works on this subject \cite{o2,o1}; for the first appearance of this name in print see \cite{LRI}.

Erd\H{o}s and Heilbronn \cite{heilbronn} conjectured that there exists an absolute constant $c$ such that
$\Ol(G)\le c \sqrt{|G|}$. This was proved by Szemer{\'e}di \cite{szemeredi}, and Olson \cite{o2} gave early results on the refined problem of determining a good (or optimal) value of the constant $c$; he obtained $c=3$.
Considerably later, Hamidoune and Z{\'e}mor \cite{H} improved this estimate
to $\Ol(G)\le \sqrt{2} \sqrt{|G|} + \varepsilon (|G|)$ where $\varepsilon(x)$ is $\mathcal{O}(\sqrt[3]{x}\log x)$, which is optimal up to the error-term as an example for cyclic groups shows.
Indeed, Subocz \cite{sub} conjectured that among all finite abelian groups of a given order the Olson constant is maximal for the cyclic group of that order.

Concerning precise values of $\Ol(G)$, Deshouillers and Prakash \cite{Deshouillers} and Nguyen, Szemer{\'e}di, and Vu \cite{vanvu} recently determined it for prime-cyclic groups of sufficiently large order, and, subsequently, Balandraud \cite{eric} obtained this result for all prime-cyclic groups (for the precise value cf.~Section \ref{sec_abstractbound}).
Moreover, Gao, Ruzsa, and Thangadurai \cite{ruzsa} proved $\Ol(C_p \oplus C_p)=(p-1) + \Ol(C_p)$ for sufficiently large prime $p$ (namely, $p> 4.67 \times 10^{34}$ and this was improved to $p>6000$ by Bhowmik and Schlage-Puchta \cite{bhowmik}). In combination, with the above mentioned results on $\Ol(C_p)$ the exact value of  $\Ol(C_p \oplus C_p)$ is thus also known for large primes. In addition, they asserted that
\begin{equation}
\label{eq_GRT}
\Ol(C_n^{r}) \ge (n-1)+\Ol(C_n^{r-1}).
\end{equation}
For $p$-groups of large rank (relative to the exponent) Gao and Geroldinger \cite{GaoGe} proved that the Olson constant is equal to the Davenport constant (see Section \ref{sec_prel} for the definition) in contrast to the above mentioned results where the Olson constant is considerably smaller than the Davenport constant; in particular, for $q$ a prime power and  $r\ge 2q + 1$, it is known that $\Ol(C_{q}^r)= 1 + r(q-1)= (r-1)(q-1) +q$.
In addition, it is known that for $r \ge 2n +1$, we have $\Ol(C_{n}^r)\ge 1 + r(n-1)= (r-1)(n-1) +n$. 
Moreover, Subocz \cite{sub} determined the value of the Olson constant for groups of exponent at most three.

In addition, recently Nguyen and Vu \cite{nguvu2} obtained, as a consequence of a result on the structure of incomplete sets---this includes zero-sumfree sets---in elementary $p$-groups, that 
\[\Ol(C_p^3)\le (2 + \varepsilon) p\] 
for each $\varepsilon > 0$ and prime $p\ge p_{\varepsilon}$; the very recent work of Bhowmik and Schlage-Puchta \cite{bsp2} on the structure of zero-sumfree sequences in $C_p^r$ allows them to obtain such results, too. This shows that equality in \eqref{eq_GRT} holds at least `almost' for $r=3$ and large primes. 

The results of \cite{GaoGe}, on the one hand, show that for $n$ a prime power and $r$ large relative to $q$ equality always holds in \eqref{eq_GRT}.
On the other hand, as pointed out in \cite{GaoGe1}, they also show that equality cannot always hold in \eqref{eq_GRT}, as this would imply $\Ol(C_n)\ge n$, which is not true except for $n\in \{1,2\}$. A problem that remained open up to now is whether or not equality in \eqref{eq_GRT} holds at least for all sufficiently large primes for fixed (small) $r>2$, say $r=3$.

We briefly discuss the contributions of this paper. In this initial discussion, we focus on our contributions for the important special case of (elementary) $p$-groups; our actual investigations are carried out in more generality.

On the one hand, we determine the precise value of the Olson constant for
a larger class of $p$-groups with large rank; roughly, we can replace the condition $r \ge 2q +1$ mentioned above by $r\ge q$ (note that neither the results of \cite{GaoGe} nor ours are limited to homocylic groups, for our precise result see in particular Corollary \ref{lr_co_pgr} and Theorem \ref{thm_homocyclic}; note in our results we use $\SD_1(G)$ to denote $\Ol(G)$, for details see Section \ref{sec_inv}).

On the other hand, we obtain a lower bound for $\Ol(C_p^r)$ for any rank---and various other types of groups---that  `smoothly' interpolates between the two extreme scenarios $ r\ge p$  and $r\le 2$, improving on the existing lower bound for the case of `medium size' rank (see, in particular, Corollary \ref{lr_co_gen} and again Theorem \ref{thm_homocyclic}); the only lower bound known is the one obtained by repeated application of \eqref{eq_GRT} in combination with the bound for cyclic groups, which approximately
yields for $C_n^r$ the lower bound $(r-1)(n-1) + \sqrt{2n}$ while we obtain approximately
$(r-1)(n-1) + \min \{r,n\} + \sqrt{\max \{0, 2(n-r)\} }$.
 In particular, our construction offers
an explanation for the difference between the case of large rank and small rank groups, and shows that equality in \eqref{eq_GRT} fails to hold (already) for $r=3$ for all but finitely many primes (possibly for all but the prime $2$).

In addition to the Olson constant we investigate a closely related but distinct constant called the Strong Davenport constant (introduced by Chapman, Freeze, and Smith \cite{CFS}). We determine its exact value for several new types of $p$-groups, including groups of the form $C_p$ and $C_p^2$ (see Section \ref{sec_large}) as well as groups with large rank (see Corollary \ref{lr_co_pgr}).

Indeed, recasting the problem of determining the Olson constant in a suitable way allows to investigate these two constant in a unified way.
We defer a detailed discussion of the Strong Davenport constant to Section \ref{sec_inv}.

\section{Preliminaries}
\label{sec_prel}

For clarity, we fix our notation regarding standard notions, and
introduce and recall some specific terminology and notation.
In particular, we give a detailed account of all notations related
to sets and sequences as our notation regarding `sets' is somewhat unorthodox,
yet convenient for the present context.

\subsection{General notation}
We denote by $\mathbb{N}$ and $\mathbb{N}_0$ the positive and
non-negative integers, resp.
We denote by $[a,b]=\{z \in \mathbb{Z} \colon a \le z \le b\}$
the interval of integers.

We use additive notation for abelian groups.
For $n \in \mathbb{N}$, let $C_n$ denote a cyclic group of order $n$.
Let $G$ be a finite abelian group.
Then, there exist uniquely determined integers $1 < n_1 \mid \dots  \mid n_r $ such that
 $G\cong C_{n_1} \oplus\ldots \oplus C_{n_r}$.
We denote by  $\exp(G)=n_r$ (except for $|G|=1$, where the exponent is $1$) the exponent of $G$ and by $\mathsf{r}(G)=r$ the rank of $G$. Moreover, for a prime $p$, let $\mathsf{r}_p(G)=|\{i\in [1,r]\colon p \mid n_i\}|$ denote the $p$-rank of $G$.
Moreover, we set $\mathsf{D}^{\ast}(G)=\sum_{i=1}^r(n_i-1)+1$.

We call a group elementary if its exponent is squarefree, a $p$-group if the exponent is a prime power, and
homocyclic if it is of the form $C_n^r$ for $r,n \in \mathbb{N}$.
For an element $g\in G$, we denote by $\ord(g)$ its order.
For $d\in \mathbb{N}$, we denote by $G[d]\subset G$ the subgroup
of elements of order dividing $d$.

\subsection{Sequences and sets}

It is now fairly common (see, e.g., \cite{GaoGe1,barcelona,GeHaBOOK}) to consider sequences---in the context of the problems considered in the present paper---as elements of $\sequ{}(G)$ the, multiplicatively written, free abelian monoid with basis $G$. Of course, in a strict sense, these are not (finite) sequences in the traditional sense, as the terms are not ordered. Yet, this is irrelevant---indeed desirable---for the problems considered here.
We denote the identity element of $\sequ{} (G)$ by $1$ and call it the empty sequence. We refer to a divisor $T\mid S$ as a subsequence, which is compatible with usual intuition regarding this term, and we use the notation
$T^{-1}S$ to denote the unique sequence fulfilling $T (T^{-1}S)=S$, which can be interpreted as the sequence where the terms appearing in $T$ are removed (taking multiplicity into account).
Of course, every map $f:G \to G'$, for finite abelian groups $G$ and $G'$, can be extended in a unique way
to a monoid homomorphism from $\sequ{}(G)$ to $\sequ{}(G')$, which we also denote by $f$.
For $h\in G$, let $s_h: G \to G$ be defined via $g\mapsto g + h$. For $h \in G$ and $S \in \sequ{} (G)$, let  $h+S=s_h(S)$, i.e., the sequence where each
term is `shifted' by $h$.

Let $S\in \sequ{} (G)$, i.e., $S= \prod_{g\in G}g^{v_g}$ with $v_g\in \mathbb{N}_0$ and there exist up to ordering uniquely determined $g_1, \dots, g_{\ell}\in G$
such that $S=g_1 \dots g_{\ell}$.
We denote by $|S|=\ell$ the length, by $\sigma(S)= \sum_{i=1}^{\ell}g_i$  the sum, by $\supp(S)=\{g_1,\dots,g_{\ell}\}=\{g \in G \colon v_g > 0\}$ the support of $S$.
Moreover, we denote for $g\in G$, by $\mathsf{v}_g(S)=v_g$ the multiplicity of
$g$ in $S$ and by $\mathsf{h}(S)=\max \{\mathsf{v}_g(S)\colon g \in G\}$ the height of $S$.

We denote by
$\Sigma(S)=\{\sigma(T)\colon 1 \neq T \mid S\}$ the set of subsums of $S$.
We call $S$ zero-sumfree if $0 \notin \Sigma(S)$, and we call it a zero-sum sequence if $\sigma(S)= 0$.
A minimal zero-sum sequence is a non-empty zero-sum sequence all of whose
proper subsequence are zero-sumfree.

We denote the set of zero-sumfree sequence by $\freesequ{} (G)$ and
the set of minimal zero-sum sequence by $\minsequ{} (G)$.

As mentioned in the introduction, we are mainly interested in zero-sumfree sets and minimal zero-sum sets. However, for our investigations it is crucial to allow for a seamless interaction of `sets' and `sequences', and to consider
sequences with various restrictions on the multiplicities of the terms.

To formalize this, we introduce the following terminology.
Let $G$ be a finite abelian group and $S \in \sequ{}(G)$.
For $\ell\in \mathbb{N}_0$, let $\cm_{\ell}(S) = \sum_{g \in G} \max \{0, \mathsf{v}_g(S)-\ell \}$ the cumulated multiplicity of
level $\ell$. We have $\cm_0(S)=|S|$ and $\cm_1(S)=|S|-|\supp(S)|$.
The most important case for our purpose is the case $\ell=1$,
thus we often simply write $\cm(S)$ for $\cm_1(S)$ and call it the cumulated multiplicity.

We set $\sequ{k}^{\ell} (G)= \{S \in \sequ{} (G)\colon \cm_{\ell}(S)\le k\}$.
Moreover, we set $\minsequ{k}^{\ell}(G) = \sequ{k}^{\ell} (G) \cap \minsequ{} (G)$, and ${}^{k}\!\mathcal{A}^{\ast  \ell} (G)= \sequ{k}^{\ell} (G) \cap \freesequ{} (G)$.
Again, if we do not write the index $\ell$, we mean $\ell= 1$.

In all formal arguments, rather than subsets of $G$ we consider elements of $\sequ{0}(G)=\sequ{0}^1(G)$, and use all notations and conventions introduced for sequences (in the sense of the present paper). In other words, we consider squarefree sequences, i.e., sequences with height at most one, rather than sets. At some points of our arguments, this admittedly causes some inconveniences, yet we prefer this over the ambiguities that would result from mixing notions for sequences and sets, or using multi-sets.

We end the preliminaries, by recalling the definition of and some results on the Davenport constant, which though not the actual subject of our investigations is of considerable relevance for our investigations (for detailed information see, e.g., \cite{GaoGe1,GeHaBOOK}).

The Davenport constant of $G$, denote $\Dav(G)$, is typically defined as the smallest integer $\ell_G$ such that each sequence $S \in \sequ{}(G)$ with $|S|\ge \ell_G$ satisfies $0 \in \Sigma(S)$, i.e., has a non-empty zero-sum subsequence.
Equivalently, it can be defined as $\max \{|S|\colon S \in \minsequ{} (G)\}$, i.e., as the maximal length of a minimal zero-sum sequence. Though, this is easy to see and very well-known, it is a non-trivial assertion;
indeed, the set-analog (cf.~Section \ref{sec_inv}) and the analog assertions when restricted to sequences over subsets of $G$ are well-known to be not true.
Moreover, the small Davenport constant, denoted $\dav(G)$, is defined as
$\max \{|S| \colon S \in \freesequ{} (G)\}$, i.e., the maximal length of a zero-sumfree sequence.
It is easy to see that $\Dav(G) = \dav(G)+1$ and $\Dav(G)\ge \Das(G)$.
Moreover, for several types of groups it is known that $\Dav(G)=\Das(G)$;
in particular, this is true for $G$ a $p$-group and for $G$ a group of rank at most $2$. In addition, it is conjectured that this equality also holds for groups of rank three and homocyclic groups. However, it is well-known that this equality does not always hold (cf.~\cite[Section 3]{GaoGe1}).

\section{The invariants}
\label{sec_inv}

In this section, we recall and introduce the invariants to be studied in the present paper; in Section \ref{sec_abstractbound} we further generalize
one of these notions, yet we defer this for the clarity of the exposition.
In addition we recall and derive some general properties of these invariants.
The main motivation for introducing these new invariants is that we need them to formulate our arguments regarding the classical invariants in an efficient way.
\begin{df}
Let $G$ be a finite abelian group and $k \in \mathbb{N}_0 \cup \{\infty\}$.
\begin{enumerate}
\item We denote by
$\SD_k(G)= \max \{ |S| \colon S \in \minsequ{k}(G) \}$ the maximal length of a minimal zero-sum sequence with cumulated multiplicity (of level $1$) at most $k$; we call it the $k$-th Strong Davenport constant.
\item We denote by $\ol_k(G)= \max \{ |S| \colon S \in \freesequ{k}(G) \}$
the maximal length of a zero-sumfree sequence with cumulated multiplicity at most $k$; we call it the $k$-th small Olson constant.
\item We denote by $\Ol_k(G)$ the smallest integer $\ell_G$ such that
each $S\in \sequ{k}(G)$ with $|S|\ge \ell_G$ has a non-empty zero-sum subsequence; we call it the $k$-th Olson constant.
\end{enumerate}
\end{df}
Since the elements of $\sequ{0}(G)$ are effectively sets,  $\SD_0(G)$ and $\ol_0(G)$ are
the maximal cardinality of a minimal zero-sum set and zero-sumfree set, respectively.
That is, the $0$-th [small] Olson constant is the classical [small] Olson constant, while
the $0$-th Strong Davenport constant is the classical Strong Davenport constant, denoted $\SD(G)$.
The notion Strong Davenport constant was introduced in \cite{CFS} (and further investigated in \cite{baliski,baginski}; yet note that in \cite{chabela}
different terminology is used). To be precise,
the definition for the Strong Davenport constant given in \cite{CFS} is
$\max \{|\supp(S)|\colon S \in \minsequ{} (G)\}$, i.e., the maximal number of distinct elements appearing in a minimal zero-sum sequence, and it is proved (see \cite{CFS}) that this quantity is always equal to $\SD_0(G)$ as defined here---it is now common to use the definition recalled here rather than
the original equivalent one (cf., e.g., \cite[Section 10]{GaoGe1}).

It is known that, in contrast to the situation for the Davenport constant,
$\SD(G)$ does not necessarily equal $\Ol(G)$; it is however known that $\Ol(G)-1 \le \SD(G)\le \Ol(G)$ (see \cite{baliski,baginski}).

First, we collect some facts on the just defined invariants,
and then we establish a relation among them that in fact shows that it is sufficient to consider the invariants $\SD_k(G)$; to highlight the connection to the classical problem and since it is useful in certain arguments, we nevertheless introduce the higher-order Olson constants.

\begin{lem}
\label{lr_lem_basic0}
Let $G$ be a finite abelian group and $k,j \in \mathbb{N}_0 \cup \{\infty\}$.
\begin{enumerate}
\item $\Ol_k(G)=\ol_k(G)+1$.
\item $\ol_k(G)=\dav(G)$ for $k \ge \dav(G)-1$, and $\SD_k(G)=\Dav(G)$ for $k \ge \Dav(G)-1$. In particular, $\ol_{\infty}(G)=\dav(G)$ and $\SD_{\infty}(G)=\Dav(G)$.
\item If $k\le j$, then $\ol_k(G)\le \ol_{j}(G)$ and $\SD_k(G)\le \SD_{j}(G)$. In particular, $\ol_k(G)\le \dav(G)$ and $\SD_k(G)\le \Dav(G)$.
\end{enumerate}
\end{lem}
\begin{proof}
1. From the very definitions we get  $\Ol_k(G) > \ol_k(G)$ and $\Ol_k(G)-1\le \ol_k(G)$.

\noindent
2. For each non-empty sequence $S$, we have $\cm(S)\le |S|-1$.
Thus, for each non-empty $S\in \freesequ{} (G)$ we have $\cm(S)\le \dav(G)-1$
and for each $S\in \minsequ{} (G)$ we have $\cm(S)\le \Dav(G)-1$.
Hence, under the respective assumption on $k$, we have $\minsequ{k} (G) = \minsequ{} (G)$ and $\freesequ{k} (G) = \freesequ{} (G)$, and the claim follows by the definitions. The additional statement is now obvious.

\noindent
3. For $k \le j$, we have $\freesequ{k} (G) \subset \freesequ{j} (G)$ and
$\minsequ{k} (G) \subset \minsequ{j} (G)$, and the claim follows.
Using 2, the additional claim follows.
\end{proof}

Now, we link the Strong Davenport constants and the Olson constants,
and establish an additional bound; an important special case is established in \cite{baliski,baginski}.

\begin{lem}
\label{lr_lem_basic}
Let $G$ be a finite abelian group and $k \in \mathbb{N}_0 \cup \{\infty\}$.
\begin{enumerate}
\item $\Ol_k(G) = \SD_{k+1}(G)$.
\item $\ol_{k+1}(G) \le \ol_k(G) + 1$.
\item $\SD_{k+1}(G) \le \SD_{k}(G) + 1$.
\end{enumerate}
\end{lem}
\begin{proof}
1. Let $S\in \freesequ{k}(G)$. Then, $-\sigma(S)S$ is a minimal zero-sum sequence and $\cm( - \sigma(S) S) \le 1 + \cm(S)\le k+1$. This implies $\Ol_k(G) = \ol_k(G) + 1 \le \SD_{k+1}(G)$.
Conversely, let $A \in \minsequ{k+1}(G)$. Then, $g^{-1}A$ is zero-sumfree for each $g \mid A$,
and we can choose $g\mid A$ in such a way that $\cm(g^{-1}A)\le k$. Thus, $\SD_{k+1}(G)-1 \le \ol_k(G)$.

\noindent
2. For $k=\infty$ the claim is trivial. We assume $k<\infty$.
Let $S \in \freesequ{k+1}(G)$ with $|S|= \ol_{k+1}(G)$.
If $S \in \freesequ{k}(G)$, the inequality follows immediately.
Otherwise, let $g \in \supp(S)$ with $\mathsf{v}_g(S)= \mathsf{h}(S)$. Then, $\cm (g^{-1}S) < \cm (S)$.
Thus $g^{-1}S \in \freesequ{k}(G)$ and $|S|-1 \le \ol_k(G)$.

\noindent
3. For $k\neq 0$ this is immediate by 1 and 2. For $k=0$, the claim is by 1 equivalent to the assertion that $\Ol_0(G)\le \SD_0(G) +1$. This is established in \cite{baliski,baginski} (also cf.~\cite{GaoGe1}).
\end{proof}

In view of the results of this section, we see that all the invariants
 $\SD(G)$, $\Ol(G)$, and $\Dav(G)$ can be expressed as particular instances of
an invariant $\SD_k(G)$, namely $\SD_0(G)$, $\SD_1(G)$, and $\SD_{\infty}(G)$, respectively.

Beyond technical advantages for our subsequent investigations, this has some
conceptual relevance, too. Namely, the fact that $\Ol(G)=\SD_1(G)$ suggests
a heuristic regarding the problem whether for a given group
$\SD(G)=\Ol(G) $ or $\SD(G)=\Ol(G)-1$ holds, a problem that so far was not well-understood.

\begin{rem}
\label{rem_heuristic}
If $\SD(G)$ and $\Ol(G)$ are (much) smaller than $\Dav(G)$, then
it is likely that $\SD(G)=\Ol(G)-1$.
\end{rem}
This is based on the reasoning that if imposing restrictions on the multiplicity has a strong effect on the maximal length of minimal zero-sum sequences fulfilling theses restrictions---documented by the fact that $\SD(G)$ is (much) smaller than $\Dav(G)$---, then relaxing these restrictions should typically already have a slight effect on the maximal length. In Section \ref{sec_large}
we prove some results that support this heuristic.
In particular, these results document that is actually only `likely'
that $\SD(G)=\Ol(G)-1$, and that there are special cases were this fails
(e.g., for $C_p$, $p$ prime, whether or not this is the case, depends on the specific $p$,
yet for almost all, in the sense of density, we have $\SD(C_p)=\Ol(C_p)-1$).

As mentioned above, our motivation for introducing $\SD_k(G)$ is mainly a technical one, to investigate
$\Ol(G)=\SD_1(G)$ and $\SD(G)=\SD_0(G)$.
However, additional investigations of these invariants could be of interest.
For example, one could ask for the smallest $k_G$ such that $\SD_{k_G}(G)=\Dav(G)$;
to determine this $k_G$ would mean, to solve a weak form of the inverse problem associated to $\Dav(G)$ (cf.~Corollary~\ref{lr_co_2r} for details).

\section{An `abstract' lower-bound construction}
\label{sec_abstractbound}

In this section, we establish a fairly flexible construction
principle for zero-sumfree sets (and sequences),
and give some first applications. More specialized investigations and additional
improvements are given in later parts of the paper.

We start by defining a more general version of $\SD_k(G)$.

\begin{df}
Let $G$ be a finite abelian group. Let $k \in \mathbb{N}_0\cup\{\infty\}$
and $\ell\in \mathbb{N}$.
We set $\SD_{(k,\ell)}(G)= \max \{|S| \colon S \in \minsequ{k}^{\ell} (G)\}$, i.e., the maximum of the length of minimal zero-sum sequences of cumulated multiplicity of level $\ell$ at most $k$.
\end{df}
Evidently, $\SD_{(k,1)}(G)= \SD_{k}(G)$.
We refrain from introducing the analogs of other notions, such as the Olson constant, in this more general setting.

The main reason for introducing these invariants, is the following technical
result. Its implications are discussed in later parts of the paper.

\begin{thm}
\label{thm_abstractbound}
Let $G$ be a finite abelian group and $H$ a subgroup of $G$. Let $k_1,k_2 \in \mathbb{N}_0\cup \{\infty \}$.
Then
\[
\SD_{k_1 + \max\{0,k_2-1\}  }(G) \ge \SD_{(k_1,|H|)}(G/H)+ \SD_{k_2}(H) -1 - \epsilon
\]
where $\epsilon = 1 $ if $k_1= 0$ and $|H| \mid \SD_{(0,|H|)}(G/H)$; and $\epsilon =0$ otherwise.
\end{thm}
As is apparent from the proof, in certain cases we can choose $\epsilon=0$ even if
$k_1=0$ and $|H| \mid \SD_{(0,|H|)}(G/H)$; we do not formalize this claim, yet encounter it in
Section \ref{sec_large}.

Before proving this result, we make some observations on
$\SD_{(k,\ell)}(G)$.
We start by collecting some general properties, which partly expand on Lemma \ref{lr_lem_basic0}.

\begin{lem}
Let $G$ be a finite abelian group. Let $k,k' \in \mathbb{N}_0\cup\{\infty\}$
and $\ell,\ell' \in \mathbb{N}$.
\begin{enumerate}
\item If $k'\le k$ and $\ell' \le \ell$, then $\SD_{(k',\ell')}(G)\le \SD_{(k,\ell)}(G)$.
\item $\SD_{(k,\ell)}(G)\le \Dav(G)$, and if $k +\ell \ge \Dav(G)$ or $\ell \ge \exp(G)$, then equality holds.
\end{enumerate}
\end{lem}
\begin{proof}
1. Since for $\ell' \le \ell$, we have $\cm_{\ell'}(S)\ge \cm_{\ell}(S)$ for each $S\in \sequ{}(G)$, it follows, for $k'\le k$, that $\minsequ{k'}^{\ell'}(G)\subset \minsequ{k}^{\ell}(G)$.  The claim follows.

\noindent
2. Since $\minsequ{k}^{\ell}(G)\subset \minsequ{}(G)$, and under the
imposed conditions equality holds, the claim follows.
\end{proof}
If $G$ is not cyclic, the condition $\ell\ge \exp(G)$,
can be weakened to $\ell \ge \exp(G)-1$.

In order to get explicit lower bounds from Theorem \ref{thm_abstractbound},
we need lower bounds for $\SD_{(k,\ell)}$ for cyclic groups.
We recall a well-known construction and comment
on its quality (cf.~\cite{eric,vanvuH}).

\begin{lem}
\label{lem_standardbound}
Let $n,\ell \in \mathbb{N}$ and $k\in \mathbb{N}_0 \cup \{ \infty \}$.
\begin{enumerate}
\item If $k + \ell \ge n$, then $\SD_{(k,\ell)}(C_n)=n$,
\item Suppose $k\le n-1$.
Let $d=\bigl \lfloor \frac{-1 + \sqrt{1 + 8(n-k)/\ell}}{2}\bigr \rfloor$.
Then,
\[\SD_{(k,\ell)}(C_n)\ge k + \ell d + \left\lfloor\frac{n-k}{d+1}-\frac{\ell d}{2}\right \rfloor.\]
In particular, $\SD_k(C_n)\ge k  + \bigl \lfloor \frac{-1 + \sqrt{1 + 8(n-k)}}{2}\bigr \rfloor$.
\end{enumerate}
\end{lem}
\begin{proof}
Let $e$ be a generating element of $C_n$.

\noindent
1. It suffices to note that $\cm_{\ell}(e^n) = n - \ell \le k$.

\noindent
2. Let $d' =\bigl \lfloor\frac{n-k}{d+1}-\frac{\ell d}{2}\bigr \rfloor$.
We note that $d$ is the largest element in $\mathbb{N}_0$ such that $\ell d(d+1)/2 \le n-k$ and we
let $d'$ denote the largest element in $ \mathbb{N}_0$ such that $d'(d+1)\le n-k-\ell d(d+1)/2$.
We consider the sequence $S=e^{k} (\prod_{i=1}^{d}(ie))^{\ell}((d+1)e)^{d'}$.
Since $k + \ell d(d+1)/2 + d' (d+1)\le n$ it follows that $S$ is a minimal zero-sum sequence or zero-sumfree. Moreover, $\cm_{\ell}(S)\le k$.
Let $je\mid S$ with $j$ maximal, and set
$S' = (-\sigma((je)^{-1}S))(je)^{-1}S$. Then $S'$ is a minimal zero-sum sequence with $\cm_{\ell}(S')\le \cm_{\ell}(S)$; note that $-\sigma((je)^{-1}S)) = j'e$ with $j' \in [j,n]$.
Thus, $\SD_{(k,\ell)}(C_n)\ge |S'|$, and the claim follows.
The additional statement follows, noting that $d'=0$ in view of $\ell=1$.
\end{proof}

The lower bound for $\SD_{(k,\ell)}(C_n)$ is not always optimal (see Section \ref{sec_large} for additional discussion). Yet, for $(k,\ell)=(1,1)$ it is always close to the true value (cf.~Section \ref{sec_intro}), and this should be the case for all other  $(k,\ell)$, too.
Indeed, in certain cases the optimality is known, i.e., the constant actually
is equal to the lower bound. For $n$ prime and $k\neq 0$ this follows by a recent result of Balandraud \cite{eric} (to see this,
consider an extremal minimal zero-sum sequence and remove an element with maximal multiplicity, and apply \cite[Theorem 8]{eric} to the resulting zero-sumfree sequence; the conditions in the above lemma correspond to the extremal case of
that result). And, for large $k+\ell$ (namely, for $k+\ell \ge \lfloor n/2  \rfloor +2$, and even slightly below this value), this is
a direct consequence of a result on the structure of long minimal zero-sum sequences, see \cite{chen,yuan} and also \cite[Section 5.1]{barcelona} for an exposition.
For a detailed investigation of the case $k=0$ and $\ell=1$, i.e. $\SD(C_n)$, and prime $n$,  see Section \ref{sec_large}; in this case equality at the lower bound does not always (though, most of the time) hold.

For the proof of Theorem \ref{thm_abstractbound},
as well as in some other arguments,
we need the following result that expands on \cite[Lemma 7.1]{GaoGe}.

\begin{lem}
\label{lr_lem_sum}
Let $G$ be a finite abelian group, $g \in G$, and $k \in [1,|G|-1]$.
There exists some $S \in \sequ{0}(G)$ with $|S|=k$ and $\sigma(S)=g$,
except if all of the following conditions hold: $\exp(G)=2$, $k\in \{2, |G|-2\}$, and $g=0$.
\end{lem}
For the sake of completeness and since the later is relevant in some special case, we discuss the two remaining meaningful values of $k$, namely $k=0$ and $k=|G|$.
For $k=0$, clearly, the only element of $G$ that can be represented as a sum of $k$ elements is $0$. For $k=|G|$, the problem reduces to determining the sum of all elements of a finite abelian group: it is $0$ if the $2$-rank of $G$ is
not $1$, and the unique element of order $2$ if the $2$-rank is $1$ (see, e.g., \cite{GaoGe} for a detailed argument).
Moreover, note that the formulation of the exceptions is sharp.

\begin{proof}
We first assume that $G$ is an elementary $2$-group.
For $|G|=2$, we have $k=1$ and the claim is obvious. Thus, assume $|G|\neq 2$.
We recall that the sum of all elements of $G$ is $0$ (cf.~above).
First, we address the problem for $g=0$.
By \cite[Lemma 7.1]{GaoGe} there exists for each $\ell \in [0,|G|/2 -1]\setminus \{2\}$ some $T \in \sequ{0}(G)$ with $|T|=\ell$ and  $\sigma(T)=0$; let $T' \in \sequ{0}(G)$ the element with $\supp(T')=G \setminus \supp(T)$, then $|T'|=|G|-\ell$ and $\sigma(T')=0$ (recall that the sum of all elements of $G$ is $0$).
And, letting $T'' \in \sequ{0}(G)$ denote an element such that $\supp(T'')$ is a subgroup of index $2$ of $G$, we get $|T''| = |G|/2$ and $\sigma(T'')=0$, except in case $|G|=4$, yet this situation is covered by the exceptional case.

Now, let $g \neq 0$.  By the above reasoning, we know that there exists a $T\in \sequ{0}(G)$ with $\sigma(T)=0$ and $|T|=k-1$ except for $k=3$ and $k=|G|-1$ (we address these two cases later). Clearly, $T$ cannot contain all non-zero elements and thus we may assume $g \nmid T$, applying a suitable automorphism of $G$. Then $gT \in \sequ{0}(G)$ has the claimed property.
For $k=|G|-1$ we can take $\prod_{h \in G \setminus \{g\}} h$, which has sum $0-g= g$, and
for $k=3$ we can take $0(g+h)h$ where $h \in G \setminus \{0,g\}$.

Now, suppose $G$ is not an elementary $2$-groups.
Let $G[2]$ denote the subgroup of elements of order dividing $2$, and let $r$ denote its rank, i.e., the $2$-rank of $G$. Moreover, let $G\setminus G[2] = G_1 \uplus (-G_1)$. We set $\delta=0$ or $\delta=1$ according to $k$ even or odd, respectively.

If $k \le |G| - 2^r$ and $g \neq 0$, we consider $g0^{\delta}\prod_{h\in G_0}(-h)h$
where $G_0 \subset G_1 \setminus \{-g,g\}$---note that at most one of the elements $-g$ and $g$ is contained in $G_1$---with cardinality $\lfloor (k-1)/2 \rfloor$.
Likewise, for $k \le |G| - 2^r + 1$ and $g = 0$, we consider
$0^{1-\delta}\prod_{h\in G_0}(-h)h$ where $G_0 \subset G_1$ with cardinality $ \lfloor k/2 \rfloor $.

We continue with considering the case $g=0$. Suppose $k > |G| - 2^r + 1$ (this implies $r \ge 2$).
If $k \notin \{ |G| - 2, |G| - 2^r + 2 \}$, we consider $T \prod_{h \in G_1}(-h)h$ where $T \in \sequ{0}(G[2])$
with $|T|= k - |G| + 2^r$ and $\sigma(T)=0$, which exists by the argument for elementary $2$-groups given above, since $k - |G| + 2^r$ is neither $2$ nor $|G[2]| - 2$.
For $k \in \{|G| - 2, |G| - 2^r + 2\}$ we consider
$T \prod_{h \in G_1\setminus \{h'\}}(-h)h$ where $h'$ is arbitrary (since $G$ is not an elementary $2$-group $G_1$ is non-empty) and $T \in \sequ{0}(G[2])$ with $\sigma(T)=0$ and $|T|=4$ or $|T|=2^r$, respectively (recall that $r\ge 2$).

Now, suppose $g \neq 0$ and $k > |G| - 2^r \ge 2^r$.
Let $T=\prod_{h \in G[2]\setminus \{0\}}h$.
If $r\neq 1$, then $\sigma(T)=0$.
In this case, for $g \notin G[2]$, we consider $g0^{\delta}T\prod_{h\in G_0}(-h)h$ where
$G_0 \subset G_1 \setminus \{ -g, g \}$ with cardinality $ \lfloor (k-2^r)/2 \rfloor $.
And, for $g \in G[2]$ (note that in this case $r\neq 0$), we consider  $0^{\delta}g^{-1}T\prod_{h\in G_0}(-h)h$ where $G_0 \subset G_1$ with cardinality $\lfloor (k+2-2^r)/2 \rfloor $; note that
for $k=|G|-1$ we have that $k + 2 -2^r$ is odd and thus $ \lfloor (k+2-2^r)/2 \rfloor = (|G| - 2^r)/2$.

It remains to consider $r=1$; let $e\in G$ denote the element of order $2$.
For $g \notin G[2]$, we consider $(g+e)e0^{\delta}\prod_{h\in G_0}(-h)h$ where
$G_0 \subset G_1 \setminus \{-g-e,g+e\}$ with cardinality $ \lfloor (k-2)/2 \rfloor $ (note that for $k=|G|-1$, we have that $k-2$ is odd).
And for $g=e$, we consider $e0^{\delta}\prod_{h\in G_0}(-h)h$ where $G_0 \subset G_1$ with cardinality $ \lfloor (k-1)/2 \rfloor $.
\end{proof}

Now, we give the proof of Theorem \ref{thm_abstractbound}

\begin{proof}[Proof of Theorem \ref{thm_abstractbound}]
The special cases $H=G$ and $H=\{0\}$ are trivial; we thus exclude them from our further investigations.

Let $\pi : G \to G/H$ denote the canonical map, as well as its extension to the monoid of sequences.
Let $T \in \minsequ{k_1}^{|H|} (G/H)$ with length $\SD_{(k_1,|H|)}(G/H)$,
and let $T_2 \in \minsequ{k_2}(H)$. Moreover, let $h_0\in \supp( T_2)$ an element with maximal multiplicity.

We consider the collection of sequences $\sequ{}_T=\pi^{-1}(T)\cap \sequ{k_1}(G) \subset \sequ{}(G)$; where here $\pi: \sequ{}(G)\to \sequ{}(G/H)$ as detailed in Section \ref{sec_prel}.
As each element $\overline{g} \in G/H$ has $|H|$ preimages under $\pi$, the set
$\sequ{}_T$ is non-empty.
For each $S \in \sequ{}_T$ we have that $\sigma(S) \in H$, yet the sum of each proper and non-empty subsequence
is not an element of $H$.

Suppose there exists some $S_0 \in \sequ{}_T$ with $\sigma(S_0)=h_0$.
Then, $A=S_0(h_0^{-1}T_2)$ is a minimal zero-sum sequence with $\cm(A)= \cm(S_0)+ \max \{0,\cm(T_2)-1\}=k_1 + \max\{0,k_2-1\}$---note that the supports of $S_0$ and $T_2$ are disjoint---and $|A|= |S_0| + |T_2|-1 = |T| + |T_2|-1$, establishing the inequality with $\epsilon=0$.

It thus remains to establish the existence of such a sequence $S_0$, and a closely related sequence for the case
$\epsilon = 1$.

First, suppose $k_1>0$.
Let $S_0' \in \sequ{}_T$ and let $h_0' = \sigma(S_0')\in H$.
Let $g_0'\in \supp(S_0')$ with maximal multiplicity.
We set $g_0 = g_0' - h_0' + h_0$ and $S_0= g_0(g_0')^{-1}S_0'$.
Then $S_0 \in \sequ{}_T$ and $\sigma(S_0)=h_0$.

Now, suppose $k_1=0$.  Suppose there exists some $\overline{g} \in G/H$ such
that $0 < v=\mathsf{v}_{\overline{g}}(T)< |H|$.
Let $T_0=\overline{g}^{-v}T$ and $S_0' \in \pi^{-1}(T_0)\cap \sequ{0}(G)$; moreover, let $g \in \pi^{-1}(\overline{g})$.
We set $h_0'= vg + \sigma(S_0')$.

We assert that we may assume that there exists some $F \in \sequ{0}(H)$ with $|F|=v$ and $\sigma(F)= h_0-h_0'$.
For $H$ not an elementary $2$-group this is immediate by Lemma \ref{lr_lem_sum}.
For $H$ an elementary $2$-group, possibly choosing a different sequence $T_2$, we may assume that $h_0$ is an arbitrary non-zero element of $H$, and it thus suffices that $\sigma(F)+h_0' \neq 0$, which
of course can be achieved.
Let $F$ be such a sequence and set $S_0=(g+ F)S_0'$. Then $\cm(S_0)= 0$ and $\sigma(S_0)= h_0$, and the claim follows as discussed above.

It remains to consider the case that $k_1= 0$ and $\mathsf{v}_{\overline{g}}(T) \in \{0, |H|\}$
for each $\overline{g} \in G/H$, in particular $|H| \mid \SD_{(0,|H|)}(G/H)$.
Since $\supp(T)$ can not be a group, there exist, possibly equal, $g_1, g_2 \in \supp(T)$ such that $g_1 + g_2 \notin \supp(T)$.
We set $T' = (g_1+g_2)g_1^{-1}g_2^{-1}T$; for the case $g_1=g_2$, recall that $\mathsf{v}_{g_i}(T)=|H| \ge 2$.
We have that $T' \in \minsequ{0}^{|H|} (G/H)$ and there exists some $\overline{g} \in G/H$ such
that $0 < v=\mathsf{v}_{\overline{g}}(T')< |H|$.

We can now apply the same argument with $T'$ instead of $T$ to get the lower bound with $\epsilon= 1$---note that $|T'|= |T|-1$---except in the case $|T'|=1$, since (the analog) of $S_0$ and $T_2$ might not have disjoint support.
Note that in this exceptional case, we also have $|H|=2$, and thus $|G|= 4$. As in this case,
the right-hand side is at most $2= \Dav(G/H)+\Dav(H)-2$, while the left-hand side is at least $\SD_0(G)\ge 2$,
the assertion is also true in this exceptional case.
\end{proof}

We point out some special cases contained in this result (for a special case of the first assertion see \cite{baginski}, and the third assertion was obtained by a different argument by Baginski \cite{baliski}).
\begin{co}
Let $G$ be a finite abelian group and $H$ a subgroup.
\begin{enumerate}
\item $\Ol(G)\ge \Ol(G/H) + \Ol(H) - 1$.
\item $\SD(G)\ge \SD(G/H) + \Ol(H) - 2 \ge \SD(G/H) + \SD(H) - 2$.
\item If $|G|\ge 3$ and $H$ is a proper subgroup, then $\SD(G) > \SD(H)$.
\end{enumerate}
\end{co}
\begin{proof}
The first two assertions are clear by Theorem \ref{thm_abstractbound} with $k_1=k_2$ equal to $1$,
and with $k_1=0$ and $k_2=1$, respectively.

It remains to prove the final assertion.
Since $|G|\ge 3$, implies that $\SD(G)\ge 2$, we can assume that
$|H|\ge 3$ (as $\SD(H)=1$ for $|H|\le 2$).
If $\SD_{(0,|H|)}(G/H)\ge 3$, the claim is immediate by Theorem \ref{thm_abstractbound},
and if $\SD_{(0,|H|)}(G/H)=2$, the claim also follows by \ref{thm_abstractbound} in view of
$|H|\nmid \SD_{(0,|H|)}(G/H)$. Since $|H|\ge 3$, we have
$\SD_{(0,|H|)}(G/H)\ge 3$ for $|G/H|\ge 3$, it remains
to consider $|G/H|=2$. Yet, in this case, $\SD_{(0,|H|)}(G/H)=2$, and
as again $|H|\nmid \SD_{(0,|H|)}(G/H)$, the claim follows.
\end{proof}
Applying Theorem \ref{thm_abstractbound} with $k_1= k_2 = \infty$
we can also obtain the classical estimate $\Dav(G)\ge \Dav(G/H)+ \Dav(H)-1$.
We frequently make use of the following special case of Theorem \ref{thm_abstractbound}.

\begin{co}
\label{lr_lem_maintech}
Let $H$ be a finite abelian group, $k\in \mathbb{N}_0 \cup \{ \infty \}$, and $n \in \mathbb{N}\setminus \{1\}$.
Then, for each $m\in [1,n]$ we have,
\[\SD_{k+\delta}(H\oplus C_n) \ge \SD_{k+1}(H)+m-1\]
where $\delta= \max \{0, m - |H| +1\}$.
In particular, $\SD_k(C_n\oplus C_n)\ge \SD_{k+1}(C_n)+n-2$.
\end{co}
\begin{proof}
We apply Theorem \ref{thm_abstractbound} with $k_2= k+1$ and $k_1 = \delta $ to get
$\SD_{k+\delta}(H\oplus C_n) \ge \SD_{(\delta, |H|)}(C_n) + \SD_{k+1}(H)-1-\epsilon$, where $\epsilon$ is as defined there.

If $\epsilon=1$, we have $\delta=0$ and  $|H|\mid \SD_{(\delta,|H|)}(C_n)$. So,
 $\SD_{(\delta,|H|)}(C_n)-2 \ge |H|-2\ge m-1$, the last inequality by the fact $m-|H|+1\le \delta\le 0$.

We thus may assume $\epsilon=0$. It thus suffices to show that $\SD_{(\delta,|H|)}(C_n)\ge m$.
For $e$ a generating element of $C_n$, the sequence
$S=e^{m-1}((n-m+1)e)$ is a minimal zero-sum sequence and $\cm_{|H|}(S) \le \max \{0,m- |H|\}\le \max \{0,m- |H|+1\}$, implying that $\SD_{(\delta, |H|)}(C_n)\ge m$.

The additional claim follows applying the result with $m=n-1$.
\end{proof}

We end this section with some discussion of Corollary \ref{lr_lem_maintech}.
First, we mainly apply this result in situations where for $m=n$ we still have $\delta=0$; thus,
we also disregard the fact that the result can be improved for $m<n$. Moreover, in Section \ref{sec_refined} we obtain similar results, where imposing stronger assumptions better results can be obtained.
Second, in the special case that the $2$-rank of $G$ is $1$ and there exists an extremal sequence with respect to $\SD_k(G)$ containing the element of order $2$ more than once, we can replace $\max \{0, n - |G|+1\}$ by $\max \{0, n - |G|\}$; the reason is apparent from the proof and cf.~the remark after Lemma \ref{lr_lem_sum}.
Of course, this situation is extremely special and in fact only occurs if
$|G|=2$.
Thus, we disregard this slight improvement for our general considerations.  Yet, we still mention it, since this phenomenon is responsible for the special role of elementary $2$-groups (cf.~Section \ref{sec_small}).

Finally, in the case $C_n \oplus C_n$, we note that for $n$ and $k>0$, we can get another lower bound, which can be better, for
$\SD_{k}(C_n^2)$, namely
\[
\SD_{k}(C_n^2)\ge \SD_{k}(C_n)+n-1
\]
that is obtained by applying the above result with $m=n$---note that in this case $\delta=1$---and $k-1$.
However, in case $\SD_{k+1}(C_n)> \SD_k(C_n)$, the inequalities coincide and yield
two distinct ways to construct a zero-sum sequence in $\minsequ{k}(C_n^2)$ of the same length.
Indeed, it follows from the work of Bhowmik and Schlage-Puchta \cite{bhowmik}, as well as of Nguyen and Vu \cite{nguvu2}, for the Olson constant, i.e., $k=1$, that for $n$ a large prime (at least $6000$ suffices) one obtains all elements of maximal length of $\minsequ{1}(C_n^2)$ using these two constructions (or only the latter one, if $\SD_{k+1}(C_n) = \SD_k(C_n)$), and elements of maximal length
from $\minsequ{2}(C_n)$ and $\minsequ{1}(C_n)$, respectively.
More explicitly, for $n>6000$ prime and $A \in \minsequ{1}(C_n^2)$ of maximum length, i.e., length $\SD_1(C_n^2) = \Ol(C_n^2)$, 
there exists an independent
generating set $\{ e_1, e_2\}$
such that $A= (e_1 + S)e_2^{-1}T$ where $S \in \sequ{1}( \langle e_2 \rangle )$
with $|S|=n$ and $\sigma(S)=e_2$, and
$T\in \minsequ{1}(\langle e_2 \rangle)$ with $e_2\mid T$ (an element of maximal multiplicity) and $|T|=\SD_1(C_n)$,
or $A=(2e_1+ae_2)(e_1+S)e_2^{-1}T$ where
$S \in \sequ{0}( \langle e_2 \rangle)$
with $|S|=n-2$ and $\sigma(S)=(1-a)e_2$, and
$T\in \minsequ{2}(\langle e_2 \rangle)$ with $e_2\mid T$ (an element of maximal multiplicity) and $|T|=\SD_2(C_n)$;
where the latter case only occurs if $\SD_2(C_n)> \SD_1(C_n)$.

\section{Basic results for groups of large rank}
\label{sec_basic}

In this section, we discuss how Corollary \ref{lr_lem_maintech} can
be used to obtain good, and in certain cases optimal, lower bounds
for $\Ol(G)$ and $\SD(G)$ for groups with large rank (in a relative sense).

Informally and restricted to the case of sets, the basic idea and novelty of this construction, which is encoded in this result, is
to use minimal zero-sum sequences---not sets---and zero-sumfree sequences over subgroups of $G$ to construct
minimal zero-sum sets and zero-sumfree sets over $G$ (the idea to extend `sets to sets', and also `sequences to sequences', are both frequently used, e.g., in \cite{ruzsa} to obtain the lower bound mention in Section~\ref{sec_intro}).

Our goal in this section is to show that, using this idea, lower bounds, which seem to be fairly good, can be obtained in a direct way.
In Section \ref{sec_refined}, we discuss other, more involved, variants of this approach that in certain cases yield slightly, though perhaps significantly, better bounds.

In the following result, we give one type of lower bound that can be derived using this method.

\begin{thm}
\label{lr_thm_maintech}
Let $G = \oplus_{i=1}^r C_{n_i}$ where $r\in \mathbb{N}\setminus \{1\}$ and $1 < n_1 \le \dots \le n_r$. Let $k \in \mathbb{N}_0$.
If $t \in [2,r]$ such that $n_{s} < \prod_{i=1}^{s-1} n_i$ for each $s \in [t+1,r]$, then the following assertions hold.
\begin{enumerate}
\item If $n_{t-1}\neq n_t$, then
\[\SD_k(\oplus_{i=1}^{r}C_{n_i}) \ge \sum_{i\in [1,r]\setminus \{t\}} (n_i -1)+ \SD_{k+r-1}( C_{n_t}).\]
\item If $k+r >2$, then
\[\SD_k(\oplus_{i=1}^{r}C_{n_i}) \ge \sum_{i\in [1,r]\setminus \{t\}} (n_i -1)+ \SD_{k+r-2}( C_{n_t}).\]
\end{enumerate}
\end{thm}
We point out that although in this result conditions are imposed these are mild assumptions.
In particular, note that the choice $t=r$ is always admissible.
Thus, except for $G= C_n \oplus C_n$ and $k=0$---for this case see the remarks after Corollary \ref{lr_lem_maintech}---this result can actually be applied.

We break up the proof of this result into several partial results. On the one hand, we do so for the clarity of the exposition. Yet, on the other hand, these partial results are of some independent interest, since they can be applied and combined in different ways.

As already mentioned, the basic idea is to apply repeatedly Corollary \ref{lr_lem_maintech}.
Obviously, we want to do this in such a way that the lower bound for $\SD_k(G)$ that we obtain at the end is
as large as possible.
Yet, the optimal strategy is not always obvious.
On the one hand, we want to keep the $\delta$s small, in the best case equal to $0$, which suggests to `split off' small cyclic components.
On the other hand, it is better to `split off' large cyclic components, since
$\SD_{\ell}(C_m)$ for fixed $\ell$ is closer to $\Dav(C_m)$ if $m$ is small, and it is thus better if the cyclic group that finally remains is as small as possible.

The former strategy is formalized in Proposition \ref{lr_prop_small}. And, the later in Proposition \ref{lr_prop_large}, restricted to groups where it is well applicable. Theorem \ref{lr_thm_maintech} is a combination of these two basic strategies with one parameter $t$ to balance them.
As we see in Section \ref{sec_refined} additional refinements are possible.

\begin{pro}
\label{lr_prop_small}
Let $r\in \mathbb{N}\setminus \{1\}$ and $n_1, \dots , n_r \in \mathbb{N}$ with $1< n_1 \le \dots \le n_r$.
Then, for each $s\in [0, r-2]$ and in case $n_{r-1}\neq n_r$ in addition for $s=r-1$,
\[\SD_k(\oplus_{i=1}^{r}C_{n_i}) \ge \sum_{i=1}^s (n_i -1)+ \SD_{k+s}(\oplus_{i=s +1}^r C_{n_i}).\]
Moreover, if $k+r >2$, then at least
\[\SD_k(\oplus_{i=1}^{r}C_{n_i}) \ge \sum_{i=1}^{r-1} (n_i -1)+ \SD_{k+r-2}( C_{n_r}).\]
\end{pro}
\begin{proof}
We prove the result by induction on $s$.
For $s=0$ there is nothing to show.
Thus, we assume that the assertion holds for some $s\ge 0$, and all $r$ and $n_i$, and show the claim for $s+1$.
We consider $\SD_k(\oplus_{i=1}^{r} C_{n_i})$. By Corollary \ref{lr_lem_maintech}, we
know, provided that $n_1 < \prod_{i=2}^{r}n_i$,
\begin{equation}
\label{lr_eq_prop_small}
\SD_k(\oplus_{i=1}^{r} C_{n_i}) \ge (n_1 - 1) + \SD_{k+1}(\oplus_{i=2}^{r} C_{n_i}).\end{equation}
This is only not the case if $r=2$ and $n_1 = n_2$.
In this case, we have again by Corollary \ref{lr_lem_maintech} if $k>0$, applying it with $k-1$, that
$\SD_k(\oplus_{i=1}^{r} C_{n_i}) \ge (n_1 - 1) + \SD_{k}(\oplus_{i=2}^{r} C_{n_i}) $, and are done.
We consider $\SD_{k+1}(\oplus_{i=2}^{r} C_{n_i})$. If $r=2$, there remains nothing to show. Otherwise, we apply the induction hypothesis, to get
$\SD_{k+1}(\oplus_{i=2}^{r} C_{n_i})\ge \SD_{k+1 + s}(\oplus_{i=s+2}^{r} C_{n_i})$
if $s+1 \in [0,r-2]$ and also for $s+1 = r-1$ if $n_{r-1}\neq n_r$. Moreover, for $s+1=r-1$ in case $(k+1)+(r-1)>2$ we get $\SD_{k+1}(\oplus_{i=2}^{r} C_{n_i})\ge \sum_{i=2}^{s+1}(n_i - 1) + \SD_{k+s}(\oplus_{i=s+2}^{r} C_{n_i})$.
In combination with \eqref{lr_eq_prop_small}, this yields our claim.
\end{proof}

\begin{pro}
\label{lr_prop_large}
Let $H$ be a finite abelian group. Let $1<n_1 \le \dots \le n_r$ positive integers such that for each $s \in [1,r]$ we have $n_s < |H|\prod_{i=1}^{s-1}n_i$. Let $G=H \oplus (\oplus_{i=1}^r C_{n_i})$ and let $k \in \mathbb{N}$. Then,
\[\SD_k(G)\ge \SD_{k+r}(H)+ \sum_{i=1}^r (n_i-1).\]
\end{pro}
\begin{proof}
We induct on $r$.
For $r=1$, this is merely Corollary \ref{lr_lem_maintech}; note that by assumption $n_1<|H|$ and thus $\delta=0$.
And, the induction-step is shown by applying first the induction hypothesis with $n_1, \dots,n_{r-1}$
and then noting that, since $n_r < |H|\prod_{i=1}^{r-1}n_i$ we can apply Corollary \ref{lr_lem_maintech}, again with $\delta=0$.
\end{proof}

The proof of Theorem \ref{lr_thm_maintech} is now merely a combination of the preceding results.

\begin{proof}[Proof of Theorem \ref{lr_thm_maintech}]
Suppose $t$ fulfills our assumption. Let $H = \oplus_{i=1}^t C_{n_i}$.
First, we apply Proposition \ref{lr_prop_large} with $H$ as just defined---the conditions are fulfilled by our assumption on $t$---to get 
\(\SD_k(H \oplus (\oplus_{i=t+1}^r C_{n_i})) \ge \SD_{k+(r-t)}(H) + \sum_{i=t+1}^r(n_i-1).\)
We then apply Proposition \ref{lr_prop_small} to the group $H$
 to get
\[\SD_{k+r-t}(H) \ge \SD_{(k+r-t)+t -2 + \epsilon}(C_{n_t}) +\sum_{i=1}^{t-1}(n_i-1)\]
with $\epsilon =1$ if $n_{t-1}\neq n_t$ and with  $\epsilon=0$ if $k+r > 2$.
Combininig these two inequalities, yields the result. 
\end{proof}

In the remainder of this section, we discuss several applications of Theorem \ref{lr_thm_maintech}.

First, we give a simple explicit lower bound. For simplicity of the presentation, we ignore certain improvements and impose the conditions that the rank of $G$ is at least three.
In important special cases, we give a more detailed analysis later.

\begin{co}
\label{lr_co_gen}
Let $G = \oplus_{i=1}^r C_{n_i}$ where $r\in \mathbb{N}\setminus \{1,2\}$ and $1 < n_1 \mid \dots \mid n_r$. Let $k \in \mathbb{N}_0$.
If $t \in [2,r]$ such that $n_{s} < \prod_{i=1}^{s-1} n_i$ for each $s \in [t+1,r]$, then
\[\SD_k(G) \ge 
\Das(G) -  \max \left \{0,n_t - k-r+2  -\left\lfloor \frac{-1 + \sqrt{1 + 8(\max\{0,n_t-k-r+2\})}}{2}\right\rfloor \right  \}.\]
In particular, if $r \ge n_t +1 -k$, then $\SD_k(G)\ge \Das(G)$.
\end{co}
\begin{proof}
By Theorem \ref{lr_thm_maintech}, we know that
$\SD_k(\oplus_{i=1}^{r}C_{n_i}) \ge \sum_{i\in [1,r]\setminus \{t\}} (n_i -1)+ \SD_{k+r-2}( C_{n_t})$.
By Lemma \ref{lem_standardbound}, we know that $\SD_{k+r-2}( C_{n_t}) \ge k+r-2  +\lfloor (-1 + \sqrt{1 + 8(n-(k+r-2))})/2 \rfloor$ for $k+r-2  \le n_t-1$ and $\SD_{k+r-2}( C_{n_t})=n_t$ for
$k + r - 2 \ge n_t $. The claims follow.
\end{proof}

 The following result improves and generalizes results in \cite{GaoGe} and \cite{baginski}. We recall that $\SD_0(G)=\SD(G)$ and $\SD_1(G)=\Ol(G)$ and point out that the main point is that we get the exact value of these classical constants for a larger class of groups (not the fact that we also get the values of $\SD_k(G)$ for other $k$).
\begin{co}
\label{lr_co_pgr}
Let $G$ be a $p$-group, or more generally a finite abelian group with $\Dav(G)=\Das(G)$, and $k\in \mathbb{N}_0$.
If $\mathsf{r}(G)\ge \exp(G)+1-k$, then
\[\SD_k(G)= \Das(G).\]
\end{co}
\begin{proof}
We have $ \SD_k(G) \le \Dav(G) = \Das(G) $ by Lemma \ref{lr_lem_basic0} and assumption.
It thus suffices to show that if $ \mathsf{r}(G) \ge \exp(G) + 1 - k $, then
\(\SD_k(G)\ge \Das(G).\)
By Lemma \ref{lem_standardbound}, we have $\SD_{k+\mathsf{r}(G)-2}(H)= |H|$ for every cyclic subgroup $H$ of $G$.
The result is thus clear for $\mathsf{r}(G)=1$ and follows by Theorem \ref{lr_thm_maintech}, with $t=\mathsf{r}(G)$,
in all other cases (note that $\mathsf{r}(G)=2$ and $k=0$ are impossible by the condition).
\end{proof}

For certain types of groups additional improvements are possible.
On the one hand, we did not use our method in its full strength, e.g., we could weaken the condition $\mathsf{r}(G)\ge \exp(G)+1-k$ to $\mathsf{r}(G)\ge n_t+1-k$ with $n_t$ as in Theorem \ref{lr_thm_maintech}.
In particular, in this way we see that for certain $p$-groups we can assert the equality of
the Olson constant and Strong Davenport constant with the Davenport constant even if the rank is much smaller than the exponent; a particularly extreme example, where the rank is only of order $\log\log (\exp(G))$ are groups of the form $C_p^2\oplus  (\oplus_{i=0}^{r-3}C_{p^{2^i}})$ (note that we can chose $t=3$).
On the other hand, there are more subtle improvements for particular types of groups (see Section \ref{sec_refined}).

Our results can also be used in a somewhat different direction. Namely, they can be used to show the existence of finite abelian groups---we could also exhibit explicit examples---where the Olson and the Strong Davenport constant
exceed the $\Das$-invariant (by any prescribed value). To the best of our knowledge, no example of a group
with  $\SD(G)> \Das(G)$ or $\Ol(G)> \Das(G)$ appeared in the literature up to now.

\begin{co}
\label{co_excess}
Let $k \in \mathbb{N}_0 \cup \{\infty \}$ and $d \in \mathbb{N}$. There exists some finite abelian group $G$ such that
\[\SD_k(G) \ge \Das(G)+d.\]
\end{co}
\begin{proof}
It is well-known that there exists some finite abelian group $G'$ such
$\Dav(G') \ge \Das(G')+ d$; in fact, this follows directly from the fact that there exists a
$G''$ with $\Dav(G'') \ge \Das(G')+ 1$ and considering $(G'')^d$ (cf.~\cite[Section 3]{GaoGe1}).
Note that $G'$ is non-cyclic. By Lemma \ref{lr_lem_basic0}, we know that
$\Dav(G')= \SD_{k'}(G')$ for some, in fact each, sufficiently large $k'$, and we assume $k' \ge k$.
Let $n= \exp(G')$. By Proposition \ref{lr_prop_large} we have $\SD_k(G' \oplus C_n^{k' - k})\ge \SD_{k'}(G')+ (k' -k)(n-1)= \Dav(G')+(k' -k)(n-1) \ge \Das(G')+d+ (k' -k)(n-1) = \Das(G' \oplus C_n^{k' - k})+d$.
\end{proof}

Recasting this result in a negative way, we see that even taking various improvements presented in later parts of the paper into account, our method cannot yield the actual value of the Olson and the Strong Davenport constant for \emph{all} groups, as none of the explicit bounds exceeds $\Das(G)$.

In the result below we characterize for groups of rank at most two for which $k$ we have $\SD_k(D)=\Dav(G)$. We do so mainly to illustrate the relation to the inverse problem associated to the Davenport constant and thus keep the proof brief; the result is a direct consequence of the recent solution of the inverse problem associated to $\Dav(G)$ for groups of rank two,
due to Reiher \cite{reiher}, Geroldinger, Gao, Grynkiewicz \cite{GGG}, and Schmid \cite{schmid}.

\begin{co}
\label{lr_co_2r}
Let $G$ be a group of rank at most two and $k \in \mathbb{N}_0$.
Then $\SD_k(G)= \Dav(G)$ if and only if
\begin{itemize}
\item $k\ge \exp(G)-1$ for $G$ homocyclic and not of the form $C_2^2$,
\item $k\ge \exp(G)-2$ otherwise.
\end{itemize}
\end{co}
\begin{proof}
We assert that $\SD_k(G)= \Dav(G)$ under the respective assumptions.
For $G$ of the form $C_2^2$, the claim follows by direct inspection (cf.~Section \ref{sec_small} for a result including this case). Assume $G$ is not of that form.
We recall that $\Dav(G)=\Das(G)$.
Since under the assumptions on $k$, we have $\SD_k(G)\ge \Das(G)$ (by Lemma \ref{lr_lem_basic0}, Corollaries \ref{lr_co_pgr} and \ref{lr_lem_maintech} for the cyclic, rank two homocyclic, and remaining case, resp.),
we see that $\SD_k(G)= \Dav(G)$ holds for the claimed $k$.

Conversely, let $A \in \minsequ{} (G)$ with $|A|=\Dav(G)$. We have to show that
$\cm(S)\ge \exp(G)-1$ and $\cm(S)\ge \exp(G)-2$, resp.
For cyclic $G$ it is well-known that $A=e^{\exp(G)}$ for $e$ a generating element of $G$ and the claim follows in this case. For $G$ of rank two, the structure of $S$ is known as well; \cite[Corollary]{GGG} gives a conditional result and in \cite{reiher} it was proved that this condition is always fulfilled.
We discuss the homocylic case.
Let $n=\exp(G)$.
In this case it is known that there exists an independent generating set $\{e_1,e_2\}$ such that $A=e_1^{n-1}\prod_{i=1}^n(e_2+a_ie_1)$
with $a_i\in [0,n-1]$ and $\sum_{i=1}^n a_i \equiv 1 \pmod{n}$.
This last condition implies that, for $n\neq 2$, not all $a_i$s are distinct (cf.~the discussion after Lemma \ref{lr_lem_sum}).
Thus $\cm(A)\ge (n-2)+1$, implying the claim.
The non-homocyclic case is similar; we omit the details.
\end{proof}

\section{Groups with large exponent}
\label{sec_large}

We discuss the problem of obtaining lower bounds for the Olson and the Strong Davenport constant
for groups with a large exponent (in a relative sense).
In particular, our considerations include cyclic groups and groups of rank two.
In our general discussion, we emphasize the Olson constant over the Strong Davenport constant. At the end, we determine the Strong Davenport constant for
$C_p$ and $C_p^2$, for $p$ a large prime, where the Olson constant is already known.

For the Olson constant, the case of cyclic groups received, in particular recently, considerable attention (cf.~the preceding discussions).
As discussed the case of prime cyclic groups is meanwhile solved;
yet, the general case remains open and it is known that the answer cannot be as uniform
as in the prime case. We show how our method of constructing zero-sumfree sets allows
to derive all known `exotic' examples of large zero-sumfree sets in a fairly direct way.

To simplify the subsequent discussion, we recall the two classical constructions for large zero-sumfree sets of cyclic groups; for the former also see Lemma \ref{lem_standardbound}.
Let $n\in \mathbb{N}_{\ge 4}$ and $e$ a generating element of $C_n$.
\begin{itemize}
  \item $\prod_{i=1}^k ie$, where $\sum_{i=1}^k i \le n-1$, is an element of $\freesequ{0}(C_n)$.
  \item $(-2e)e\prod_{i=3}^k ie$, where $\sum_{i=1}^k i \le n+1$, is an element of $\freesequ{0}(C_n)$ for $n \ge 4$.
\end{itemize}
Note that in both cases $k$ is the length/cardinality, while the condition on $k$ in the second case is weaker,
though for most $n$ the condition is only formally weaker.
As already discussed, for $n$ prime, these construction when choosing $k$ as large as possible yield zero-sumfree sets of maximal cardinality. However, for certain $n$, better constructions are known; they are attributed to
Selfridge in \cite[C15]{guy}.

We apply Theorem \ref{thm_abstractbound} to obtain lower bounds for groups of the form $G' \oplus C_n$ where $n$ is large relative to $|G'|$; in the result below there is no explicit assumption on the size of $n$ relative to $|G'|$, yet the result, in particular the explicit lower bounds, are only useful if $n$ is at least $|G'|$.
One way to do so, yields the following result that generalizes and refines a result for groups of rank two established in \cite[Theorem 9]{chabela}.

\begin{pro}
Let $G \cong G'\oplus C_n$ where $\exp(G')=m$ and $m \mid n$.
\begin{enumerate}
\item Then
\(\Ol(G)\ge \Ol(G')+ \SD_{(1,|G'|)}(C_n)-1.
\)
In particular,
\[\Ol(G)\ge \Ol(G')+|G'|d+  \left\lfloor \frac{n-1}{d+1} -\frac{|G'|d}{2}\right\rfloor\]
where $d= \bigl\lfloor\frac{-1+\sqrt{1+8(n-1)/|G'|}}{2} \bigr\rfloor$.
\item Then
\(\Ol(G)\ge \Ol(G'\oplus C_m)+ \SD_{(1,|G'|m)}(C_{n/m})-1.
\)
In particular,
\[\Ol(G)\ge \Ol(G' \oplus C_m) + |G'| m d + \left\lfloor \frac{n - m}{m(d+1)} -\frac{|G'|md}{2}\right\rfloor\]
where $d=\bigl \lfloor\frac{-1+\sqrt{1+8(n-m)/(m^2|G'|)}}{2}\bigr \rfloor$.
\end{enumerate}
\end{pro}
\begin{proof}
We apply Theorem \ref{thm_abstractbound} with $k_1=k_2=1$ and chose $H$ such that $H \cong G'$ and $G/H \cong C_n$,
and $H \cong G' \oplus C_m$ and $G/H \cong C_{n/m}$, respectively.
The additional statements follow by invoking Lemma
\ref{lem_standardbound}.
\end{proof}

In case $G$ is cyclic, that is $|G'|=1$, the first assertion yields no new insight,
yet the second one can yield an improvement over the classical lower bounds.
Namely, if $2 \mid n$, then this bound yields $\Ol(C_n) \ge 2 + 2 d + \lfloor \frac{n-2}{2(d+1)} -d \rfloor $ with $d = \lfloor \frac{-1 + \sqrt{2n-3}}{2} \rfloor$
and this bound can be better
than the bound $\Ol(C_n)\ge 1 + \lfloor\frac{-1 +  \sqrt{8n+9}}{2} \rfloor$
that is obtained from the second classical construction, first noted by Selfridge.
For example, note that
for every even $n$ with $(t^2+3)/2\le n < (t^2 + t -2)/2$ for some odd $t\in \mathbb{N}$, the former bound
is $t+1$, while the latter (classical) bound is only $t$.

For non-cyclic $G$, it seems that the construction state in the first part of the proposition is typically
better, and possibly never worse, than the one state in the second part.
In view of the fact that the latter contains $\Ol(G'\oplus C_m)$ while the former contains $\Ol(G')$
a precise and general comparison of the two bounds is difficult.

We continue by pointing out that there are also other ways to apply Theorem \ref{thm_abstractbound} for these types of groups.
One other way is to apply Theorem \ref{thm_abstractbound} also
with $H=G'$ and $G/H \cong C_n$, yet with $k_1=0$ and $k_2=2$. This yields
\[\Ol(G)\ge \SD_2(G')+ \SD_{(0,|G'|)}(C_n)-1-\epsilon
\]
with $\epsilon = 1$ if $|G'| \mid \SD_{(0,|G'|)}(C_n)$ and $\epsilon = 0$ otherwise.

This construction can in fact be better. For example for
$C_3 \oplus C_{12}$ it yields $9$, while the other one only yields $8$;
also note that $9$ is the actual value of the Olson constant in this case
(see \cite{sub}).

Yet, for $C_3\oplus C_{18}$ still a different way to apply
Theorem \ref{thm_abstractbound} yields a better bound; namely,
with a subgroup $H\cong C_6$ such that $G/H \cong C_9$ and
$k_1 =k_2=1$---note $\SD_1(C_6)=4$---we get the lower bound $4+8-1=11$.

Using one of these constructions, yields
the exact value of $\Ol(G)$ for all groups of rank two up to order $55$
as computed by Subocz \cite{sub}.

Finally, we point out a case were a potential improvement to Theorem \ref{thm_abstractbound} becomes relevant (cf.~the discussion after that result); again, the construction is originally due to Selfridge.
We first state the construction explicitly and then discuss how this is related to Theorem \ref{thm_abstractbound}.

Let $n=25 k(k+1)/2$ and let $m= 5k(k+1)/2$ with $k \in \mathbb{N}$.
Let $e$ be a generating element of $C_n$, and let $ H = \{ 0, me, 2me, 3me, 4me \} $ the subgroup of order $5$.
We consider the sequence
\[A=(me)^2 (2me) \prod_{j=0}^4 \bigl ( jme + \prod_{i=1}^k ie \bigr ).\]
This is a minimal zero-sum sequence with $\cm(A)= 1$, and thus $\Ol(C_n)\ge |A|= 3 + 5k$;
again this is better than the bound obtained from the classical constructions, which is $2+ 5k$.

This bound is essentially also a particular instance of our construction principle,
namely it corresponds to 
\[ \Ol(C_n) \ge \SD_2(C_5) + \SD_{(0,5)}(C_{n/5}) - 1.\]
Yet, note that we do not get this from Theorem \ref{thm_abstractbound}, since in this case we have $\epsilon = 1$; we cannot rule out that  $5 \mid \SD_{(0,5)}(C_{n/5})$ and in fact the lower bound we use is $5k$. 
However, a more careful analysis shows that in this case the extra argument in the proof of Theorem \ref{thm_abstractbound} that is responsible for $\epsilon = 1$ is not needed. 
As for one, and thus any, non-zero element of $H$, there is a sequence in $\minsequ{0}^5(C_n/H)$ of length $5k$, the lower bound for $\SD_{(0,5)}(C_{n/5})$, such that the sum of the pre-image of this sequence, as constructed in the proof of Theorem \ref{thm_abstractbound}, has this element of $H$ as its sum.

We end this section with some discussion of the Strong Davenport constant.
On the one hand, we can of course obtain explicit lower bounds in the
same way as for $\Ol(G)$ or via the inequality $\SD(G)\ge \Ol(G)-1$
from the bound for $\Ol(G)$.
According to the heuristic we presented in Remark \ref{rem_heuristic} we expect that
often $\SD(G)= \Ol(G)-1$ rather than $\SD(G)=\Ol(G)$.

We elaborate on this point for the two types of groups with large exponent
where $\Ol(G)$ is known, i.e., $C_p$ and $C_p^2$ for prime $p$ (assuming that $p$ is large).
We start by determining $\SD(C_p^2)$; a crucial tool is the
recent solution of the inverse problem associated to
$\Ol(C_p^2)$ for large primes $p$ (cf.~the discussion after Corollary \ref{lr_lem_maintech}).

\begin{thm}
Let $p$ be a prime, and suppose $p>6000$.
Then,
\[\SD(C_p^2)= \Ol(C_p^2)-1.\]
\end{thm}
\begin{proof}
By Lemma \ref{lr_lem_basic}, it suffices to show that
$\SD(C_p^2)\neq \Ol(C_p^2)=\SD_1(C_p^2)$.
Let $A\in \minsequ{1}(C_p^2)$ with maximal length.
It suffices to show that $\cm(A)=1$ (and not $0$).

As mentioned after Corollary \ref{lr_lem_maintech}, there exists an independent
generating set $\{ e_1, e_2\}$
such that $A= (e_1 + S)e_2^{-1}T$ where $S \in \sequ{1}( \langle e_2 \rangle )$
with $|S|=p$ and $\sigma(S)=e_2$, and
$T\in \minsequ{1}(\langle e_2 \rangle)$ with $e_2\mid T$ (an element of maximal multiplicty) and $|T|=\SD_1(C_p)$,
or $A=(2e_1+ae_2)(e_1+S)e_2^{-1}T$ where
$S \in \sequ{0}( \langle e_2 \rangle )$
with $|S|=p-2$ and $\sigma(S)=(1-a)e_2$, and
$T\in \minsequ{2}(\langle e_2 \rangle)$ with $e_2\mid T$ (an element of maximal multiplicity) and $|T|=\SD_2(C_p)$;
where the latter case only occurs if $\SD_2(C_p)> \SD_1(C_p)$.

Assume first that $A$ is of the former form.
Then $\cm(A)$ cannot be $0$, as this would imply $\cm(S)=0$, which contradicts
$\sigma(S)\neq 0 $ (cf. the remark after Lemma \ref{lr_lem_sum}).

Now, assume $A$ is of the latter form.
If $\cm(A)=0$, it follows that $\cm(T)=1$. Yet, then $\SD_1(C_p)\ge |T|$ and
thus $\SD_2(C_p)=\SD_1(C_p)$, rendering this case obsolete.
\end{proof}
The condition $p>6000$ stems directly from the result of Bhowmik and
Schlage-Puchta \cite{bhowmik}, in case an analogous assertion should hold, which is likely,
without that assumption, then we could drop this assumption here as well.
Only note, that $p=2$ is a special case---see the remark
after Corollary \ref{lr_lem_maintech}---and $\SD(C_2^2)=\Ol(C_2^2)$.

In the case of cyclic groups,
even prime cyclic groups, the problem is more complicated.
Indeed, whether $\SD(C_p)$ is equal to $\Ol(C_p)$ or $\Ol(C_p)-1$,
varies with $p$, as we see below. Though, typically $\SD(C_p)$ equals  $\Ol(C_p)-1$.
This is in line with the heuristic mentioned in Remark \ref{rem_heuristic}, that it is likely that $\SD(C_p)$ equals  $\Ol(C_p)-1$,
in view of the fact that $\Ol(C_p)$ and $\Dav(C_p)$ are far apart.

Our argument below crucially relies on the recent solution of the inverse problem
associated to $\Ol(C_p)$ in the detailed form given by Deshouillers and Prakash \cite{Deshouillers}, and
also see \cite{vanvu}; the fact that such results are only known for sufficiently
large primes, is the main reason we have to impose this condition.
Yet, there actually are some isolated phenomena for very small primes, on which we comment after the proof.

\begin{thm}
Let $p$ be a sufficiently large prime.
If $p=k(k+1)/2 -2$ or $p=k(k+1)/2 -4$, for some $k \in \mathbb{N}$, then $\SD(C_p)=\Ol(C_p)$.
Otherwise $\SD(C_p)= \Ol(C_p)-1$.
\end{thm}
It is widely believed that infinitely many primes of the form $k(k+1)/2 -2$ and $k(k+1)/2 -4$ exist,
as there is no `obvious' reason for the respective polynomial not to take a
prime value infinitely often (in particular, they are irreducible over $\mathbb{Q}$); we checked that there are many.

\begin{proof}
Let $S\in \freesequ{0}(C_p)$ with $|S|=\Ol(C_p)-1$.
If we can assert that $-\sigma(S)\mid S$, then $\SD(C_p)< \Ol(C_p)$.
Assume that $-\sigma(S)\nmid S$.

We distinguish various cases according to the structure of $S$ as described in \cite{Deshouillers}
(in particular, see Theorem 27 and Table 1).
According to this result, there exists a (unique) generating element $e$ of $C_p$ such that
$S=S' S''$ where
$S' = \prod_{i=1}^{\ell'} (j_i' e)$ with $\ell' \le 2$ and $j_i' \in [-4, -1]$, and
$S'' = \prod_{i=1}^{\ell} (j_i e)$ with $j_i \in [1, p/2]$ and  $\sum_{i=1}^{\ell}j_i \le p+2$; we assume $j_i<j_{i+1}$. Following \cite{Deshouillers} we write $s''$ for $\sum_{i=1}^{\ell}j_i$.

We discussed the various cases that can arise according to this classification.

\noindent
Case 1. $|S'|=0$.\\
Case 1.1. $s'' \le p-1$. Without restriction assume that $s''$ is minimal (among all $S$ in this case, fulfilling $-\sigma(S)\nmid S$).
This minimality assumption implies that $j_{\ell +1} = p - s'' > j_{\ell}$, otherwise we could
consider $(j_{\ell+1}e) \prod_{i=1}^{\ell-1} (j_i e)$ and violate the minimality assumption.
Moreover, it follows that $j_{i}=i$ for each $i$, since otherwise
 replacing $j_i$ by $j_i - 1$ for a suitable $i$ would yield a zero-sum free set, contradicting the maximality
of $S$.
Yet, this implies that $p = (\ell+1)(\ell+2)/2$, which is not the case for sufficiently large $p$.

\noindent
Case 1.2.  $s'' = p$. This is clearly impossible.

\noindent
Case 1.3. $s'' = p+1$. It follows that $e \nmid S$.
We have $p+2 = 1 + s'' \ge (\ell+1)(\ell+2)/2 > p-1$.
Since for sufficiently large $p$, we have that $(\ell+1)(\ell+2)/2$ is neither $p$ nor $p+1$, as the arising polynomials are reducible (see \cite{eric}),
it follows that $p+2 = (\ell+1)(\ell+2)/2$.
So, we get that $S=S''=\prod_{i=1}^{\ell} (i+1)e$ and $-\sigma(S)=-e$.
Note that for $p+2= (\ell+1)(\ell+2)/2$, we indeed have $\Ol(C_p)= \ell + 1$.

\noindent
Case 1.4. $s'' = p+2$. It follows that $2e \nmid S$.
We have $p+4 = 2 + s'' \ge (\ell+1)(\ell+2)/2 > p-1$.
Again, for sufficiently large $p$, $(\ell+1)(\ell+2)/2$ is neither $p$, $p+1$, nor $p+3$,  
and the case $p+2$ was already settled 
it follows that $p+4 = (\ell+1)(\ell+2)/2$.
So, we get that $S=S''=e\prod_{i=3}^{\ell} ie$ and $-\sigma(S)=-2e$.
Note that for $p+4= (\ell+1)(\ell+2)/2$, we indeed have $\Ol(C_p)= \ell + 1$.

\noindent
Case 2. $S'=(-e)$ and $s'' \le p -1$.
We consider $(-\sigma(S))S''$, and are in Case 1.

\noindent
Case 3. $S'=(-e)$ and $s'' \in \{ p , p+1,p+2\}$ is impossible; for $p$ and $p+1$ this is obvious, and
for $p+2$ note that $-\sigma(S)= (-e)$.

\noindent
Case 4. $S'=(-2e)$, and $s'' \le p -1$ or $s'' = p +1$.  Considering $(-\sigma(S))S''$, we are in  Case 1.

\noindent
Case 5. $S'=(-2e)$ and $s'' \in \{ p , p+2 \}$. This is impossible.

\noindent
Case 6. $S'=(-3e)$ or $S'=(-4e)$.  We would get that $(-\sigma(S))S''$ violates the condition $s'' \le p+2$ (where
$s''$ now is computed for $(-\sigma(S))S''$).

\noindent
Case 7. $S'=(-2e)(-e)$ or  $S'=(-3e)(-e)$. We would get that $(-\sigma(S))S''(-e)$ violates the condition $s'' \le p+2$
(now with $s''$ for $(-\sigma(S))S''(-e)$).
\end{proof}
In the proof we used that there are no sufficiently large primes of the form
$k(k+1)/2$, $k(k+1)/2 -1$, and $k(k+1)/2 -3$. However, there are some primes of this form, namely $2$, $3$, $5$, and $7$.
For these direct inspection shows that
$1=\SD(C_2)<\Ol(C_2)=2$, $\SD(C_3)=\Ol(C_3)=2$, $2=\SD(C_5)<\Ol(C_5)=3$, and
$3=\SD(C_7)<\Ol(C_7)=4$. A reason why, say, $\SD(C_7)<\Ol(C_7)$,
while our proof (see Case 1.4)  suggests that $\SD(C_p)=\Ol(C_p)$ for primes of the form $k(k+1)/2 -3$ is the fact that $7$ is so small that the
construction $(-2e)e(3e)(5e)... (ke)$ (with $k=5$ is this case),
does not yield a set.
However, for the next prime $11$, we already have $\SD(C_{11})=\Ol(C_{11})=5$;
where the lower bound is given by $(-2e)e(3e)(4e)(5e)$.

\section{Refined bounds}
\label{sec_refined}

In this section, we present two more specialized constructions
that allow to improve, in certain cases, the results established
in Section \ref{sec_basic}. In particular,
these constructions are applicable in the important case of homocyclic groups.
Indeed, we focus on this case; yet, we formulate our technical results in more generality.

Informally, the idea is similar to the one presented in Section \ref{sec_basic}, in particular see Corollary \ref{lr_lem_maintech};
yet, rather than `adding' only a cyclic component, we `add'
groups of rank two and three, resp., which can yield better results
than `adding' a cyclic component two times or three times, resp.

At first glance, the improvements might seem minimal, and perhaps not even
worth the additional effort. Yet, as we detail at the end of this section and
in Section \ref{sec_small}, this small improvement is of significance (perhaps it is even crucial).

We start with the result `adding' three cyclic components at once;
though it is more complex on a technical level, it is more
in line with the already explored construction from a conceptual point of view.

\begin{pro}
\label{pro_cn4}
Let $G$ be a finite abelian group with $\exp(G)\ge 3$, $k \in \mathbb{N}_0 \cup \{\infty\}$,
and $n_1,n_2,n_3 \in \mathbb{N}\setminus \{1,2\}$ with $n_i \le |G|+1$.
Then
\[
\SD_k(G \oplus C_{n_1}\oplus C_{n_2}\oplus C_{n_3})\ge \SD_{k+3}(G)+(n_1-1)+(n_2-1)+(n_3-1).
\]
\end{pro}

\begin{proof}
Let $\{e_1,e_2,e_3\}$ be an independent generating set of $C_{n_1}\oplus C_{n_2}\oplus C_{n_3}$ with $\ord(e_i)=n_i$ for each $i$.
Let $\pi_i: G \oplus \langle e_1,e_2,e_3\rangle \to \langle e_i\rangle$, for $i \in [1,3]$ and $\pi: G \oplus \langle e_1,e_2,e_3\rangle \to \langle e_1,e_2,e_3\rangle$ denote the standard epimorphism (subject to this decomposition and generating set).
Let $A\in \minsequ{k+3}(G)$ with $|A|\ge 3$; note that by our assumption on
$G$ such an $A$ always exists.
Let $g_1g_2g_3\mid A$ such that
$\cm((g_1g_2g_3)^{-1}A)\le k$.

Let $S_1,S_2,S_3 \in \sequ{0}(G)$ with $|S_i|=n_i-2$ and
$\sigma(S_1) = g_1$, $\sigma(S_2)= -g_1-g_2$, and
$\sigma(S_3) = -g_1 - g_2 + g_3$;
these exists by Lemma \ref{lr_lem_sum}, since $n_i-2\le |G|-1$ and $G$ is not an
elementary $2$-group.

Let
\[ \begin{split}F=& (e_1+S_1)(e_2+S_2)(e_3+S_3)(e_1+e_3)(e_1-e_3)\\ & (e_1+e_2+g_2)(-e_1+e_2+g_1+g_2)(e_2+e_3+g_1+g_2)(-e_2+e_3).
\end{split}
\]
We note that $|F|=n_1+n_2+n_3$ and $\sigma(F)=g_1+g_2+g_3$.
Thus,
\[(g_1g_2g_3)^{-1}AF\]
has sum $0$, length $|A|-3 + n_1 + n_2 + n_3$, and its cumulative multiplicity is at most $k$.
To establish our claim it remains to show that $(g_1g_2g_3)^{-1}AF$
is a minimal zero-sum sequence.
We observe that to show this it suffices to show that
for each $1\neq B \mid F$ with $\sigma(\pi(B))=0$, we have
$\sigma(B)\in \Sigma(g_1g_2g_3)$, and $\sigma(B)=g_1+g_2+g_3$ if and only if
$B=F$. Note that, since $g_1g_2g_3$ is a subsequence of a minimal zero-sum sequence, we have that $\sigma(T)\neq g_1+g_2+g_3$ for each non-empty and proper subsequence of $g_1g_2g_3$.

We first determine all minimal zero-sum subsequences of
\[\pi(F) =
e_1^{n_1-2} e_2^{n_2-2} e_3^{n_3-2}(e_1+e_3)
(e_1-e_3)(e_1+e_2)(-e_1+e_2)(e_2+e_3)(-e_2+e_3).\]
Let $C\mid \pi(F)$ be a minimal zero-sum subsequence, and let
 $C_+$, $C_-$ denote the subsequence of elements of the form $e_i+e_j$ and $e_i-e_j$ (where $i \neq j$), respectively.
We note that $C_+C_-$ is non-empty. Moreover, it follows that
\begin{equation}
\label{eq_cn4}
\pi_i(\sigma(C_+C_-))\neq e_i \text{ for } i\in [1,3].
\end{equation}

We distinguish cases according to $|C_-|$. Throughout, the argument
below let $i,j,k$ be such that $\{i,j,k\}=\{1,2,3\}$; note that
not for each choice of $i,j,k$ all sequences below exist.

\noindent
Case 0: $|C_-|=0$. Then, by \eqref{eq_cn4}, $|C_+|=3$ and
$C=e_1^{n_1-2} e_2^{n_2-2} e_3^{n_3-2}(e_1+e_3)(e_1+e_2)(e_2+e_3)=C_0$.

\noindent
Case 1: $|C_-|=1$. For $C_-=-e_i+e_j$, we get that $e_i+e_j\mid C_+$ or $e_j+e_k\mid C_+$. In the former case, it follows that $C=(-e_i+e_j)(e_i+e_j)e_j^{n_j-2}=C_1^j$.
In the latter case, it follows that $C=(-e_i+e_j)(e_j+e_k)(e_k+e_i)e_j^{n_j-2}e_k^{n_k-2}=C_1^{j,k}$.

\noindent
Case 2: $|C_-|=2$. We have $C_-=(-e_i+e_j)(-e_j+e_k)$. It follows that,
to fulfill \eqref{eq_cn4} for $k$,
$e_k+e_i\mid C_+$ or $e_k+e_j\mid C_+$.

In the latter case,  it follows that $e_j+e_i\mid C_+$, to fulfill \eqref{eq_cn4} for $j$; and $e_i +e_k\nmid C_+$. Thus we get
$C = (-e_i + e_j) (-e_j + e_k) (e_k + e_j) (e_j + e_i) e_j^{n_j-2} e_k^{n_k-2}=C_2^{j,k}$.
While, in the former case it follows that
$C= (-e_i+e_j)(-e_j+e_k)(e_i+e_k)e_k^{n_k-2}=C_2^k$.

\noindent
Case 3: $|C_-|=3$. We have $C_- = (-e_1+e_2) (-e_2+e_3) (e_1-e_3)$. This is
a zero-sum sequence. So, $C= C_-=C_3$.

This completes our classification of minimal zero-sum sequences.

Let $1\neq B \mid F$ and $B \neq F$ such that $\pi(B)$ is a zero-sum sequence.
We assert that $\pi(B)$ or $\pi(B^{-1}F)$ is a minimal zero-sum sequence.
Assume not. Then $\pi(F)=A_1\dots A_m$  with minimal zero-sum sequences
$A_i$ where $m \ge 4$. Since each $A_i$ is either equal to $C_0$
or contains at least one element of $C_-$, it follows in view of $|C_-|=3$ that, say, $A_1=C_0$. Yet, this implies that $C_3=A_1^{-1}\pi(F)=A_2 \dots A_m$, a contradiction.

Thus, since $\sigma(F)=g_1+g_2+g_3$ and in view of the comments above, in order to show that
for each $1 \neq B \mid F$ with $\sigma(\pi(B))=0$ and $B \neq F$, we have
$\sigma(B) \in \Sigma(g_1g_2g_3) \setminus \{ g_1 + g_2 + g_3 \}$, it suffices to show that $\sigma(D) \in \{ g_1, g_2, g_3, g_1 + g_2, g_1 + g_3, g_2 + g_3 \}$ for each $D\mid F$ such that $\pi(D)$ is a minimal zero-sum sequence.

We observe that for each minimal zero-sum sequence of $\pi(F)$
in view of our classification (note that each such sequence
 contains $e_i$ either with multiplicity $n_i-2$ or $0$)
there exist a unique subsequence of $F$ whose image under $\pi$ is this minimal
zero-sum subsequence.

Hence, it suffices to check explicitly the condition
$\sigma(D)\in \{g_1,g_2,g_3, g_1+g_2,g_1+g_3 ,g_2+g_3\}$
for the unique preimage of each of the minimal zero-sum
sequences of $\pi(F)$ that we determined explicitly.

We omit the details of  this routine computation.
\end{proof}

We continue with the result `adding' two cyclic components.

\begin{pro}
\label{pro_cn3}
Let $G$ be a finite abelian group with $\exp(G)\ge 3$, $k \in \mathbb{N}_0 \cup \{\infty\}$,
and $n_1,n_2 \in \mathbb{N}\setminus \{1,2\}$ with $n_i \le |G|+1$.
Then
\[
\SD_k(G \oplus C_{n_1}\oplus C_{n_2})\ge \SD_{k+2}^{2}(G)+(n_1-1)+(n_2-1)
\]
where $\SD_{k+2}^{2}(G)$ is the maximum length of a sequence $A\in \minsequ{k+2}(G)$ such that there exist distinct $g,h\in G$ such that $\cm((gh)^{-1}A)\le k$.\end{pro}
\begin{proof}
Let $\pi$ denote the canonical projection
from $G \oplus C_{n_1}\oplus C_{n_2}$ to $C_{n_1}\oplus C_{n_2}$.
Let $\{e_1, e_2\}$ be an independent generating set of $C_{n_1}\oplus C_{n_2}$ with $\ord(e_i) = n_i$.

Let $A\in \minsequ{k+2}(G)$ such that there exist distinct $g,h\in G$ with $\cm((gh)^{-1}A)\le k$; by our assumption on $G$ such a sequence exists.

For $i \in [1,2]$, let $S_i \in \sequ{0}(G)$ with $|S_i|= n_i-2$ and $\sigma(S_1)=g$ and $\sigma(S_2)=0$ (these exist by Lemma \ref{lr_lem_sum}).
Now, let $F = (e_1+S_1)(e_2+S_2)(e_1+e_2)(e_1+e_2+h-g)(e_1-e_2)(-e_1+e_2+g)$.

We note that $\cm((gh)^{-1}AF)\le k$ and $\sigma((gh)^{-1}AF)= 0$.
It thus remains to show that $(gh)^{-1}AF$ is a minimal zero-sum sequence.
Let $D \mid (gh)^{-1}AF$ be a non-empty zero-sum subsequence;
let $D=D_1D_2$ such that $D_1 \mid (gh)^{-1}A$ and $D_2 \mid F$.
We have to show that $D=(gh)^{-1}AF$.
Since $(gh)^{-1}A$ is zero-sumfree, it follows that $D_2$ is non-empty. If $D_2=F$, the claim follows,
since $\sigma((gh)^{-1}A)= -g-h = - \sigma(F)$. So, suppose $D_2 \neq F$.

We note that $\sigma(\pi(D_2))=0$.
It follows that $\pi(D_2) \in \{ (-e_1+e_2)(e_1-e_2), (-e_1+e_2)(e_1+e_2)e_2^{n_2-2}, (e_1-e_2)(e_1+e_2) e_1^{n_1-2},
(e_1+e_2)^2 e_1^{n_1 - 2} e_2^{n_2 - 2}\}$.
This implies that $\sigma(D_2) \in \{g, h\}$; note that $D_2$ is determined by $\pi(D_2)$ up to
at most one element, namely $(e_1+e_2)$ and $(e_1+e_2+g)$.
Yet, this yields a contradiction to $\sigma(D_1)= - \sigma(D_2)$, since $-g,-h\notin \Sigma((gh)^{-1}A)$ as $A$ is a minimal zero-sum sequence.
\end{proof}

In the following result we summarize the implications
of these results, for determining lower bounds for the Olson and the Strong Davenport constant for homocyclic groups; for a discussion of the quality of these bounds see the following section.
Results along the lines of those established in Section \ref{sec_basic}
could be obtained as well; yet, to avoid technicalities,
we only address this important special case.

\begin{thm}
\label{thm_homocyclic}
Let $n,r\in \mathbb{N}$ and suppose $n \ge 3$ and $r\ge 4$.
Then, for each $k\in \mathbb{N}_0$,
\[\SD_k(C_n^r)\ge \SD_{k+r-1}(C_n) + (r-1)(n-1).
\]
In particular,
\[\SD_k(C_n^r) \ge
\Das(C_n^r) -  \max \left \{0,n - k-r+1  - \left\lfloor \frac{-1 + \sqrt{1 + 8(\max\{n-k-r+1,0\})}}{2} \right \rfloor \right \}
\]
Moreover, if $\Dav(C_n^r) = \Das(C_n^r)$, in particular if $n$ is a prime power,
then, for $r \ge n-k$,
\[
\SD_k(C_n^r) = \Das(C_n^r).
\]
Additionally,
\[\SD_k(C_n^3)\ge \Das(C_n^3) - \max \left\{1,n -k-2- \left\lfloor \frac{-1 + \sqrt{1+ 8\max\{n-3-k, 0\}}}{2} \right\rfloor \right\}.\]
\end{thm}
\begin{proof}
We proceed by induction on $r$.
For $r=4$, the assertion is merely Proposition \ref{pro_cn4}.
Assume the assertion holds for some $r\ge 4$.
By Corollary \ref{lr_lem_maintech}, we get that $\SD_{k}(C_n^{r+1})\ge \SD_{k+1}(C_n^r)+n-1$.
By induction hypothesis, $\SD_{k+1}(C_n^r)\ge \SD_{k+1+(r-1)}(C_n)+(r-1)(n-1)$.
Thus, the claim follows.
To get the `in particular'-statement, we use the lower bound
on $\SD_{k+r-1}(C_n)$ established in Lemma \ref{lem_standardbound}.
Finally, the `moreover'-statement follows by the just established lower bound,
which in this case is $\Das(C_n^r)$,
and the fact that $\SD_{k}(C_n^r)\le \Dav(C_n^r)=\Das(C_n^r)$; the former
inequality by Lemma \ref{lr_lem_basic0} and the latter equation by assumption.

To prove the `additionally'-statement, it suffices by Proposition \ref{pro_cn3}, using $\SD_{k+2}^{2}$ as defined there, to show that
\[\SD_{k+2}^{2}(C_n) \ge \min \left \{ n-1, 2+k+ \left\lfloor \frac{-1 + \sqrt{1+ 8\max\{n-3-k, 0\}}}{2} \right\rfloor \right\}.\]
Let $e$ be a generating element.
For $k \ge n-3$, we set $A= e(2e)e^{n-3}$,
and for $k<n-3$, we set $A=e(2e)e^k \prod_{i=1}^{\ell-1}(ie)(xe)$
where $\ell\in \mathbb{N}$ is maximal with $\ell (\ell + 1)/2 \le n-3 - k$ and $x=n - k- \ell (\ell - 1)/2$.
\end{proof}

Recalling that $\Ol(C_p^2) = (p-1) + 1 + \lfloor \frac{-1 + \sqrt{1+ 8(p-1)}}{2} \rfloor$ for prime $p>6000$ and
by the just established result,
we see that
\[\begin{split}
\Ol(C_p^3) & \ge 2(p-1) + 3+ \left \lfloor \frac{-1 + \sqrt{1+ 8(p-4)}}{2} \right\rfloor\\
  & > 2(p-1) + 1 + \left \lfloor \frac{-1 + \sqrt{1+ 8(p-1)}}{2} \right \rfloor = (p-1) + \Ol(C_p^2),\end{split}\]
showing that equality in \eqref{eq_GRT} fails already for $r=3$ for all but finitely many primes. Evidently, equality holds for $p=2$, yet this might well be the only case (this would follow, if $\Ol(C_p^2)=(p-1)+\Ol(C_p)$ for all primes, which is conceivable; for results for primes up to $7$ see the following section).

We end this section by pointing out that the construction of Proposition
\ref{pro_cn4} is also of relevance for sequences.

\begin{rem}
\label{rem_sequence}
Let $n\ge 3$ and $C_n^4 = \oplus_{i=1}^4 \langle e_i \rangle$. The sequence
\[
\begin{split}
& (e_1+S_1)(e_2+S_2)(e_3+S_3)e_4^{n-3}\\
& (e_1+e_3)(e_1-e_3)(e_1+e_2+e_4)(-e_1+e_2+2e_4)(e_2+e_3+2e_4)(-e_2+e_3),
\end{split}
\]
where $S_1,S_2,S_3 \in \sequ{}(C_n)$ with $|S_i|=n-2$ and
$\sigma(S_1) = e_4$, $\sigma(S_2)= -2e_4$, and
$\sigma(S_3) = -e_4$
is a minimal zero-sum sequence of length $\Das(C_n^r)$.
Thus, if $\Dav(C_n^r)=\Das(C_n^r)$, in particular if $n$ is a prime power, then it is a minimal zero-sum sequence of maximum length.
\end{rem}
The point of this remark is that this sequence does not
arise in the `usual' way from zero-sum sequences of maximal length in $C_n^3$;
by `usual' way, we mean that $n$ elements from a coset are `added' and one element is `removed'.

\section{Groups with small exponent,  computational results, and discussion}
\label{sec_small}

We determined the exact value of the Olson constant, and  except for a single group, the Strong Davenport constant, for groups with very small exponent in an absolute sense, namely $\exp(G)\le 5$.
We point out that for our reasoning it is not only important that the
exponent is `small,' but it is inevitable that $\exp(G)$ is a prime power.
In particular, we consider it as significantly more challenging to extend our result, say, to $\exp(G)=6$ than to $\exp(G)=7$.

For groups with $\exp(G)\le 3$ considerable parts of the result are known
(see \cite{baliski,chabela,GGS,sub}; for
partial result for groups of exponent $4$ and $5$ see \cite{baginski,sub}). We give a proof that stresses, which parts of the result follow by the methods
detailed before, and which require an additional argument; we do so even if a direct argument
would be simpler.

\begin{thm}
Let $r \in \mathbb{N}_0$.
\begin{enumerate}
\item For $r\neq 1$, we have $\Ol(C_2^r)= \SD(C_2^r)=\Das(C_2^r)=r+1$. And, $\Ol(C_2)=2$ and $\SD(C_2)=1$.
\item For $r\ge 4$, we have $\Ol(C_3^r)= \SD(C_3^r)=\Das(C_3^r)=2 r + 1$. And, 
\begin{itemize}
\item $\Ol(C_3^3)=\Das(C_3^3)=7$ and $\SD(C_3^3)=6$;
\item $\Ol(C_3^2)=4$ and $\SD(C_3^2)=3$;
\item $\Ol(C_3)= \SD(C_3)=2$.
\end{itemize}
\item For $r \ge 5$, we have $\Ol(C_5^r)=\SD(C_5^r)=\Das(C_5^r)=4r+1$.
And,
\begin{itemize}
\item $\Ol(C_5^4)=\Das(C_5^4)=17$ and $16 \le \SD(C_5^4)\le 17$;
\item $\Ol(C_5^3)=12$ and $\SD(C_5^3)=11$;
\item $\Ol(C_5^2)=7$ and $\SD(C_5^2)=6$; 
\item$\Ol(C_5)=3$ and $\SD(C_5)=2$.
\end{itemize}
\item For $G$ a finite abelian group with $\exp(G)=4$, we have $\Ol(G)=\SD(G)=\Das(G)$, with the following exceptions:
\begin{itemize}
\item $\Ol(C_4^3)=\SD(C_4^3)=9$;
\item $\Ol(C_4^2)=6$ and $\SD(C_4^2)=5$;
\item $\Ol(C_4)=3$ and $\SD(C_4)=2$;
\item $\Ol(C_2 \oplus C_4^2) = \Das(C_2 \oplus C_4^2) = 8$ and $\SD(C_2\oplus C_4^2)=7$;
\item $\Ol(C_2^2 \oplus C_4) = \Das(C_2^2 \oplus C_4) = 6$ and $\SD(C_2^2\oplus C_4)=5$;
\item $\Ol(C_2 \oplus C_4) = \SD(C_2 \oplus C_4)=4$.
\end{itemize}
\end{enumerate}
\end{thm}
We were unable to determine $\SD(C_5^4)$; we know $16 \le \SD(C_5^4) \le 17$ and the result of partial computations suggest $\SD(C_5^4)=16$. All the values in this result match one of our lower bound constructions,
except for $C_2^2$, where the particular phenomenon discussed after Corollary \ref{lr_lem_maintech} is relevant,
which we ignored on purpose.

\begin{proof}
We recall (see Lemma \ref{lr_lem_basic0}) that $\SD(G)\le \Ol(G)\le \Dav(G)$;
and for all groups appearing in this result, since they are $p$-groups,
 we have $\Dav(G)=\Das(G)$.

\noindent
1. By Corollary \ref{lr_co_pgr}, the result is clear for $r\ge 3$ for the Strong Davenport constant and for $r \ge 2$ for the Olson constant. Yet, also $\SD(C_2^2)\ge 3$, note the example $(e_1+e_2)e_1e_2$ for independent $e_1$ and $e_2$. For $C_2=\{0,e\}$ there are precisely two minimal zero-sum sequences, namely $0$ and $e^2$, and the claim follows.

\noindent
2. For $r \ge 4$, and in addition for
$r=3$ in case of the Olson constant, the claim follows again by Corollary \ref{lr_co_pgr}.
Now, the remaining cases can be solved by computation (cf.~below for details), or by the following argument.

In \cite{GGS} it is proved that $\SD(C_3^3)=6$. By Lemma \ref{lr_lem_basic} and Corollary \ref{lr_co_2r} we know that $4 \le \Ol(C_3^2)< \Dav(C_3^2)=5$. And, $2\le \SD(C_3)\le \Ol(C_3)< 3$; note that for $C_3= \{0,e,-e\}$,
the only minimal zero-sum sequences are $0$, $(-e)e$, and $(\pm e)^3$.

It remains to show that $\SD(C_3^2)=3$.
By Lemma \ref{lr_lem_basic} and $\Ol(C_3^2) = 4$ it follows that $\SD(C_3^2) \ge 3$, and it remains to show that $\SD(C_3^2)\neq 4$.
Suppose $A\in \minsequ{0}(C_3^2)$ with $|A|=4$. A certainly contains two independent elements $e_1$ and $e_2$.
Let $g, h\in C_3^2\setminus \{e_1,e_2\}$ such that $A=e_1e_2gh$. We note that $g,h\notin  \{0, -e_1, -e_2, -e_1-e_2\}$ and $\{g,h\} \neq \{- e_1 + e_2, e_1 - e_2\}$, as otherwise we would get a proper zero-sum subsequence.
So we have that $\{g,h\}$ is equal to $\{e_1+e_2, -e_1 + e_2\}$ or  $\{e_1+e_2,e_1 - e_2\}$.
Yet, neither choice yields an element of $\minsequ{0}(C_3^2)$.

\noindent
3. By Theorem \ref{thm_homocyclic} the result is clear for
$r\ge 5$, and $\Ol(C_5^4)=\Das(C_5^4)$ follows as well.
The remaining cases are handled by computation
(cf. below for details; also see \cite{sub} for $\Ol(C_5^2)$ and $\Ol(C_5)$).

\noindent
4. For $G$ with rank at least $5$ and $C_4^4$ the assertion follows by
Corollary \ref{lr_co_pgr} and Theorem \ref{thm_homocyclic}, respectively.
By Corollary \ref{lr_co_2r} we know that $\SD_2(C_2\oplus C_4)=\Das(C_2\oplus C_4)$.
Thus, it follows by Proposition \ref{lr_prop_large} that $\Ol(G)=\Das(G)$ for any group of exponent $4$ and rank at least $3$,
having $C_2 \oplus C_4$ as a direct summand, and likewise
$\SD(G)=\Das(G)$ for any group of exponent $4$ and rank at least $4$.

The remaining values are again determined by computation.
\end{proof}

We have computed the values for two additional groups.
\begin{rem} \
\begin{itemize}
\item $\SD(C_6^3)=\Ol(C_6^3)= 14$.
\item $\SD(C_7^3)=16$ and $\Ol(C_7^3)=17$.
\end{itemize}
\end{rem}

We give a brief indication how our computation were carried out; we keep this discussion brief, as the approach is very similar to that of \cite{lzs}.

Based on the ideas from \cite{lzs}, one can formulate an algorithm for the computation of the Olson and Strong Davenport constant for a finite abelian group $G$. A naive brute force search over all subsets of $G$ is infeasible, already for quite small $G$. But with a slight variant of the algorithms from \cite{lzs}---roughly speaking, constructing zero-sumfree sets recursively---, we can avoid most of the redundant checks and therefore speed up the computation dramatically. For additional details on further speeding up these types of algorithms by pre-computations, special alignment of the pre-computed data, and on the parallelization aspects, the reader is referred to \cite[Section 3]{lzs}.

All computations were performed on a SUN X4600 node with 8 QuadCore-Opteron CPUs and 256GB RAM running with up to 32 openmp threads. Here, we give a table with the computation times (for the non-trivially fast examples).
\begin{center}
\begin{tabular}{crrc}
\bfseries group & cputime (hh:mm:ss) \\
\hline
$C_4^3$ & 00:00:18 \\
$C_5^3$ & 00:00:35 \\
$C_6^3$ & 03:03:44 \\
$C_7^3$ & 262:21:20
\end{tabular}
\end{center}

In view of the various lower bound constructions presented in this paper,
evidently the question arises how close to the true values these bounds are,
and we already gave some indications throughout the paper.
Although, we hope that the constructions presented in this paper
are flexible enough to yield good bounds in general and the exact value
for  various groups,
they obviously do not yield the exact value for all groups (see the remark after Corollary \ref{co_excess}).
And, as we saw in Section \ref{sec_large}, even in applying our method there
is considerable flexibility, so that in certain cases the
lower bounds we mentioned explicitly can be improved with the methods at hand.

Yet, for certain groups, in particular for elementary $p$-groups, neither of the above problems
arises. We state or opinion on this case in more detail below.

For rank $3$, we have some, though admittedly not too much (indeed, as we were aware of most of the explicit results before establishing the general lower bound, this evidence is even weaker), computational evidence that our constructions are optimal. We do not consider it as strong enough to justify to conjecture that the bound established in Theorem  \ref{thm_homocyclic} is sharp.  Still, we would be surprised if it were not sharp.

For rank $4$ and greater, there is not even computational evidence except for
$C_5^4$ and again we were aware of this example before coming up with our construction and, indeed, used this knowledge as motivation for our construction.
Though, we evidently hope that the bound is sharp, this is almost merely wishful thinking. A reason why we still think that this might be true, despite the fact that up to rank $4$ we encountered a new phenomenon for each rank, is the fact that the choice of the set in $C_p^3 \oplus C_p$ used in Proposition \ref{pro_cn4} is good enough to expand it to a zero-sum \emph{sequence} of maximal length in $C_p^4$ (see Remark \ref{rem_sequence}) and thus $C_p^r$ for $r\ge 4$. Thus, with this construction we overcame the `irregularity' arising from the fact that we cannot apply Corollary \ref{lr_lem_maintech} with $m=p$ for $C_p \oplus C_p$, i.e., when passing from rank one to rank two, and we hope that this was the only remaining obstacle towards a uniform formula for these invariants. As said, this hope is vague. Yet, at least we believe that if there is some $r_0$ such that there is a `uniform' formula for $\SD(C_p^r)$ and $\Ol(C_p^r)$ for all  $r\ge r_0$  (and all $p\ge p(r_0)$), then this $r_0$ is already $4$ and this formula is the lower bound we established.

 \end{document}